\documentclass[a4paper,10pt]{scrartcl}
\usepackage{amsmath,amssymb,amsthm,mathtools,xfrac,mathrsfs} 
\usepackage{psfrag}
\usepackage{epsfig}
\usepackage{color}
\usepackage{graphics}
\usepackage{multicol}
\usepackage{enumerate}
\usepackage[dvipsnames]{xcolor}
\usepackage[hyperindex]{hyperref}

\newtheorem{theorem}{Theorem}[section]
\newtheorem{lemma}{Lemma}[section]
\newtheorem{proposition}{Proposition}[section]

\newtheorem{remark}{Remark}

\mathtoolsset{showonlyrefs}

\numberwithin{equation}{section}

\newcommand{\R}{\mathbf{R}}
\def\S{\mathbf{S}}
\newcommand{\N}{\mathbf{N}}

\newcommand{\vhi}{\varphi}
\newcommand{\eps}{\varepsilon}
\newcommand{\oo}{\infty}
\newcommand{\nb}{\nabla}
\newcommand{\ov}{\overline}
\newcommand{\dw}{\downarrow}
\newcommand{\rw}{\rightarrow}

\newcommand{\longto}{\longrightarrow}
\newcommand{\Om}{\Omega}
\newcommand{\om}{\omega}
\newcommand{\restr}{ \!\!\mbox{{ \Large$\llcorner$}} }

\newcommand{\be}{\begin{equation}}
\newcommand{\ee}{\end{equation}}
\DeclareMathOperator{\supp}{supp}

\DeclareMathOperator{\argmin}{argmin}

\newcommand\wtilde{\widetilde}
\newcommand\D{\mathcal{D}}
\newcommand\lt{\left}
\newcommand\rt{\right}
\newcommand\F{\mathcal{F}}

\def\void{\varnothing}
\def\cd{\cdot}

\def\s{\sigma}
\def\g{\gamma}
\def\t{\theta}
\def\a{\alpha}
\def\b{\beta}
\def\d{\delta}

\def\x{_\xi}

\def\M{\mathcal{M}}

\def\H{\mathcal{H}}

\def\G{\mathcal{G}^d}
\def\I{\mathcal{I}}
\def\J{\mathcal{J}}
\def\L{\mathscr{L}}

\def\e{_\eps}

\def\sm{\setminus}
\def\de{\partial}

\def\dx{\;\mathrm{d}x}

\def\dt{\;\mathrm{d}t}
\def\dH{\;\mathrm{d}\H}
\def\ds{\;\mathrm{d}\s}
\def\di{\;\mathrm{d}}

\def\rwstar{\stackrel{*}{\rightharpoonup}}
\def\sm{\setminus}
\def\pt{\partial}
\def\ind{\mathbf{1}}

\title{Variational approximation of size-mass energies for $k$-dimensional currents}
\author{A. Chambolle\footnote{ CNRS, CMAP, \'Ecole Polytechnique CNRS UMR 7641, Route de Saclay, F-91128 Palaiseau Cedex France, email: antonin.chambolle@cmap.polytechnique.fr}   \and L. Ferrari\footnote{CMAP, \'Ecole Polytechnique, CNRS UMR 7641, Route de Saclay, F-91128 Palaiseau Cedex France, email: luca.ferrari@polytechnique.fr}\and B. Merlet\footnote{Laboratoire P. Painlev\'e, CNRS UMR 8524, Universit\'e Lille 1, F-59655 Villeneuve d'Ascq Cedex, France, email: benoit.merlet@math.univ-lille1.fr}}
\date{}
\begin{document}
\maketitle
\begin{abstract}
In this paper we produce a $\Gamma$-convergence result for a class of energies $\F^k_{\eps,a}$  modeled on the Ambrosio-Tortorelli functional. For the choice $k=1$ we show that $F^1_{\eps,a}$ $\Gamma$-converges to a branched transportation energy whose cost per unit length is a function $f_a^{n-1}$ depending on a parameter $a>0$ and on the codimension $n-1$. The limit cost $f_a(m)$ is bounded from below by $1+m$ so that the limit functional controls the mass and the length of the limit object. In the limit $a\downarrow0$ we recover the Steiner energy.  \\ 
We then generalize the approach to any dimension and codimension. The limit objects are now $k$-currents with prescribed boundary, the limit  functional controls both their masses and sizes. In the limit $a\downarrow0$, we recover the Plateau energy defined on $k$-currents, $k<n$. The energies $F^k_{\eps,a}$ then can be used for the numerical treatment of the $k$-Plateau problem. 
\end{abstract}


\section{Introduction}\label{Section: Introduction}
  
\medskip
Let $\Om\subset \R^n$ be a convex, bounded open  set. We consider vector measures $\s\in\M(\ov\Om,\R^n)$ of the form
\begin{equation}
\s=m\,\nu\, \H^1\restr\Sigma,
\end{equation}
where $\Sigma$ is a $1$-dimensional rectifiable set oriented by a Borel measurable tangent map $\nu:\Sigma\rw\S^{n-1}$ and $m:\Sigma\rw\R_+$ is a Borel measurable function representing the multiplicity. We write $\s=(m,\nu,\Sigma)$ for such measures. Given a cost function $f\in C(\R_+,\R_+) $ we introduce the functional
\begin{equation}
\label{calF}
\F(\s):=\begin{dcases}
\int_\Sigma f(m)\dH^1&\mbox{ if } \s=(m,\nu,\Sigma),\\
\qquad+\oo &\mbox{otherwise in } \M(\ov\Om,\R^d).
\end{dcases}
\end{equation}
Next, given $\mathscr{S}=\{x_1,\cdots,x_{n_P}\}\subset \Om$ a finite set of points and $c_1,\cdots,c_{n_P}\in\R$ such that $\sum_{j=1}^{n_P} c_j\,=\,0$, we consider the optimization problem $\F(\s)$ for $\s\in\M(\ov\Om,\R^n)$ satisfying
 \begin{equation}\label{point sources}
\nabla\cdot \sigma\, =\, \sum_{j=1}^{n_P} c_j \delta_{x_j}\qquad\mbox{ in }\D'(\R^n).
\end{equation} 
The setting is similar to the one from Beckman~\cite{MR3409718} and Xia~\cite{Xia2}. We model transport nets connecting a given set of sources $\{x_j\in\mathscr{S}\,:\, c_j>0\}$ to a given set of wells $\{x_j\in\mathscr{S}\,:\, c_j<0 \}$ via vector valued measures.   For numerical reasons, we wish to approximate the measure $\s=(m,\nu,\Sigma)$ by a diffuse object (a smooth vector field). For this, we introduce below a family of corresponding ``diffuse'' functionals ${\cal F}_{\eps,a}$ that converge towards~\eqref{calF} in the sense of \textit{$\Gamma$-convergence}~\cite{Bra1,Bra2,DalMaso}. This general idea has proved to be effective in a variety of contexts such as fracture theory, optimal partitions problems and image segmentation~\cite{Amb_Tort2,Iur,ContiFocardiIurlano,Mod_Mort}. More recently this tool has been used to approximate energies depending on one dimensional sets, for instance in~\cite{OS2011} the authors take advantage of a functional similar to the one from Modica and Mortola defined on vector valued measures to approach the branched transportation problem~\cite{Ber_Cas_Mor}. With similar techniques approximations of the Steiner minimal tree problem (\cite{Ambr_Tilli},~\cite{Gilb_Poll} and~\cite{Paol_Step}) have been proposed in~\cite{MR3337998,BonOrlOud16}. \medskip

In the present paper we first extend to any ambient dimension $n\geq 2$ the \textit{phase-field} approximation for a branched transportation energy introduced in~\cite{FerMerCha16} for $n=2$. In particular the approximate functionals ${\cal F}_{\eps,a}$ are modeled on the one from Ambrosio and Tortorelli~\cite{Amb_Tort1}.  
We also extend the construction to any dimension and co-dimension. Indeed, for $1\leq k \leq n-1$ integer, we consider $k$-rectifiable currents $\sigma=(\theta,e,\Sigma)$  where $\Sigma$ is a countably $k$-rectifiable set with approximate tangent $k$-plane defined by a simple unit multi-vector $\xi(x)=\xi_1(x)\wedge\cdots\wedge \xi_k(x)$ and $m:\Sigma\to\R_+$ is a Borel measurable function (the multiplicity). The functional~\eqref{calF} extends to $k$-currents $\sigma$  as follows,
\[
\F(\s):=\begin{dcases}
\int_\Sigma f(m(x))\dH^k&\mbox{ if } \s=(m,\xi,\Sigma),\\
\qquad+\oo &\mbox{otherwise.}
\end{dcases}
\]
\medskip

Let us define the approximate functionals and describe our main results in the case $k=1$. For our phase field approximations we relax the condition on the vector measure $\s$ replacing it by a vector field $\sigma\e\in L^2(\Om,\R^n)$. We then need to mollify condition~\eqref{point sources}. Let $\rho:\R^n\rw \R_+$ be a classical radial mollifier such that $\supp \rho\subset B_1(0)$ and $\int_{B_1(0)}\rho=1$. For $\eps>0$, we set $\rho\e=\eps^{-n}\rho(\cdot/\eps)$. We substitute for~\eqref{point sources} the condition  
\begin{equation}\label{point sources eps}
\nb\cd\sigma\e\,=\, \, \lt(\sum_{j=1}^{n_P} c_j\d_{x_j}\rt) \ast\rho\e=\sum_{j=1}^{n_P} c_j\rho\e(\cdot-x_j)\qquad \mbox{ in } \D'(\R^n).
\end{equation} 
\begin{remark}
Notice that in \eqref{point sources}~\eqref{point sources eps} the equality holds in $\D'(\R^n)$ and not only in $\D'(\Om)$ so that there is no flux trough $\de \Om$.
\end{remark}
We also consider the functions $u\in W^{1,p}(\Om,[\eta,1])$ such that $u\equiv 1$ on $\partial \Omega$ where $\eta=\eta(\eps)$ satisfies
\begin{equation}\label{Def: a}
 \eta=a\,\eps^{n}
\end{equation}
 for some $a\in\R_+$. We denote by $X\e(\Om)$ the set of pairs $(\s,u)$ satisfying the above hypotheses. This set is naturally embedded in $\M(\overline\Om,\R^n)\times L^2(\Om)$. For $(\s,u)\in\M(\overline\Om,\R^n)\times L^2(\Om)$ we set 
\begin{equation}\label{Def: Feps} 
\F_{\eps,a}(\s, u;\Om)\, :=\begin{dcases}
                         \, \int_\Om\lt[ \eps^{p-n+1}  |\nb u|^p +\dfrac{(1-u)^2}{\eps^{n-1}} +
\dfrac{u|\sigma|^2}{\eps}\rt]\dx&\mbox{if }  (\s,u)\in X\e(\Om),\\
\qquad\qquad\qquad+\oo& \mbox{in the other cases.}
                        \end{dcases}
\end{equation}
Let $X$ be the subset of $\M(\ov\Om,\R^n)\times L^2(\Om)$ consisting of those couples $(\s,u)$ such that $u\equiv 1$ and $\s=(m,\nu,\Sigma)$ satisfies the constraint~\eqref{point sources}. Given any sequence $\eps=(\eps_i)_{i\in\N}$ of positive numbers such that $\eps_i\dw0$, we show that the above family of functionals $\Gamma$-converges to 
\begin{equation}\label{Def: limit}
 \F_a(\s,u;\overline\Om)=\begin{dcases}
             \int_{\Sigma\cap\overline\Om} f_a(m(x))\dH^1(x)&\mbox{if }  (\s,u)\in X\mbox{ and }\s=m\,\nu\,\H^1\restr\Sigma,\\
             \qquad\qquad+\oo&\mbox{otherwise.}
            \end{dcases}
\end{equation} 
The function $f_a:\R_+\rw\R_+$  (introduced and studied in the appendix) is the minimum value of some optimization problem depending on $a$ and on the codimension $n-1$ (we note $f_a^d$, with $d=n-k$ in the general case $1\le k\leq n-1$). In particular we prove that $f_a$ is lower semicontinuous, subadditive, increasing, $f_a(0)=0$  and that there exists some $c>0$ such that
\begin{equation}
\label{eq1+m}
\frac{1}{c}\leq \frac{f_a(m)}{\sqrt{1+a\,m^2}}\leq c\quad \mbox{for } m>0.
\end{equation}
The $\Gamma$-convergence holds for the topology of the weak-$*$ convergence for the sequence of measures  $(\s\e)$ and for the strong $L^2$ convergence for the phase field $(u\e)$. For a sequence $(\s\e,u\e)$ we write  $(\s\e,u\e)\rw(\s,u)$ if $\s\e\rwstar\s$ and $\|u\e- u\|_{L^2}\rw0$. In the sequel we first establish that the sequence of functionals $(\F_{\eps,a})\e$ is coercive with respect this topology.
\begin{theorem}\label{teo: 1compactness}
Assume that $a>0$. For any sequence $(\s\e,u\e)\subset\M(\overline\Om,\R^n)\times L^2(\Om)$ with $\eps\dw0$, such that
\begin{equation*}
\F_{\eps,a}(\s\e,u\e;\Om)\leq F_0<+\oo,
\end{equation*} 
then there exists $\s\in\M(\ov\Om,\R^n)$ such that, up to a subsequence, $(\s\e,u\e) \to (\s,1)\in X$.
\end{theorem}
Then we prove the $\Gamma$-liminf inequality
\begin{theorem}\label{teo: 1Gammaliminf}
 Assume that $a\geq 0$. For any sequence $(\s\e,u\e)\in \M(\ov\Om,\R^n)\times L^2(\Om)$ that converges to $(\s,u)\in \M(\ov\Om,\R^n)\times L^2(\Om)$ as $\eps\dw0$ it holds
 \begin{equation*}
  \liminf_{\eps\dw0}\F_{\eps,a}(\s\e,u\e;\Om)\geq \F_a(\s,u;\ov\Om).
 \end{equation*}
\end{theorem}
We also establish the corresponding $\Gamma$-limsup inequality
\begin{theorem}\label{teo: 1Gammalimsup}
 Assume that $a\geq 0$. For any $(\s,u)\in \M(\ov\Om,\R^n)\times L^2(\Om)$ there exists a sequence $((\s\e,u\e))\subset \M(\ov\Om,\R^n)\times L^2(\Om)$ such that
  \begin{equation*}
  (\s\e,u\e)\stackrel{\eps\dw0}{\longto}(\s,u)\quad\mbox{ in }\M(\ov\Om,\R^n)\times L^2(\Om)
 \end{equation*}
 and
 \begin{equation*}\label{eq: 1limsup}
  \limsup_{\eps\dw0}\F_{\eps,a}(\s\e,u\e;\Om)\leq \F_a(\s,u;\ov\Om).
 \end{equation*}
\end{theorem}
 \medskip

As already stated, we only considered the case $k=1$ in this introduction. Section~\ref{Section: kProblem} is devoted to the extension of Theorems~\ref{teo: 1compactness}, \ref{teo: 1Gammaliminf} and \ref{teo: 1Gammalimsup} in the case where the 1-currents (vector measures) are replaced with $k$-currents.\medskip

Notice that the coercivity of the family of functionals only holds in the case $a>0$. However, as $a\dw0$ we have the important phenomena: 
\[
f_a\ \stackrel{a\dw 0}\longrightarrow\ c{\mathbf{1}}_{(0,+\oo)}\quad\mbox{pointwise},
\]
for some $c>0$. As a consequence~\eqref{Def: limit} is an approximation of $c\mathcal{H}^1(\Sigma)$ for $a>0$ small and the minimization of~\eqref{Def: Feps} in $X_\eps(\Om)$ provides an approximation of the Steiner problem associated to the set of points $\mathscr{S}$, for a suitable choice of the weights in~
\eqref{point sources}. In the case $k>1$, we obtain a variational approximation of the $k$-Plateau problem.
\medskip
  
\noindent 
\textbf{Structure of the paper:} In Section~\ref{Section: Prel&Not} we introduce some notation and recall some useful facts about vector measures and currents, we also anticipate the optimization problem defining the cost function $f_a^d$ and state some results which are proved in Appendix~\ref{appendix: Reduced Problem}. In Section~\ref{Section: 1Problem} we establish Theorems~\ref{teo: 1compactness},~\ref{teo: 1Gammaliminf} and~\ref{teo: 1Gammalimsup}. In Section~\ref{Section: kProblem} we extend these results to the case $1\leq k\leq n-1$. In Section~\ref{sec: discussion} we discuss the limit $a\dw0$. 
\section*{Acknowledgment}
The authors have been supported by the ANR project Geometrya, Grant No. ANR-12-BS01-0014-01. A.C.~also acknowledges the hospitality of Churchill College and DAMTP, U.~Cambridge, with a support of the French Embassy in the UK, and a support of the Cantab Capital Institute for Mathematics of Information. 

\section{Preliminaries and notation}\label{Section: Prel&Not}
The canonical orthonormal basis of $\R^n$ is denoted by the vectors $e_1,\dots,e_n$.  $\L^n$ denotes the Lebesgue measure in $\R^n$ and given an integer value $k$ we denote with $\om_k$ the measure of the unit ball in $\R^k$, i.e. $\mathscr{L}^k(B_1(0))$. For a point $x\in\R^n$ we note $x=(x_1;x')\in\R\times \R^{n-1}$. For any Borel-measurable set $A\subset \R^n$ we denote with $\ind_A(x)$ the characteristic function of the set $A$
\begin{equation*}
 \ind_A(x):=\begin{dcases}
             1&\mbox{ if } x\in A\\
             0&\mbox{otherwise}.
            \end{dcases}
\end{equation*} Given a vector space $Y$ and its dual $Y'$ for $\om\in Y$ and $\s\in Y'$ we write $\langle \om ,\s\rangle$ for the dual pairing. 
 
 \subsection{Measures and vector measures}\label{subSection: measureAndVec}
We denote with $\M(\Om)$ the vector space of \textit{Radon measures} in $\Om$ and with $\M(\Om,\R^n)=\M(\Om)^n$ the vector space of vector valued measures. For a measure $\mu \in \M(\Om)$ we denote with $|\mu|$ its \textit{total variation}, in the vector case $\mu\in \M(\Om,\R^n)$  we write $\mu=\nu |\mu|$ where $\nu$ is a $|\mu|$-measurale map into $\S^{n-1}$. 
We say that a measure is \textit{supported} on a Borel set $E$ if $|\mu|(\Om\sm E)=0$. For an integer $k< n $ we denote with $\H^k$ the $k$-dimensional \textit{Hausdorff} measure as in~\cite{Am_Fu_Pal}. Given a set $E\in \Om$, such that, $\H^k(E)$ is finite for some $k$ the restriction $\H^k\restr E$ defines a Radon measure in the space $\M(\Om)$. A set  $E\in \Om$ is said to be \textit{countably $k$-rectifiable} if up to a $\H^k$ negligible set $N$, $E\sm N$ is contained in a countable union of $C^1$ $k$-dimensional manifolds.

\subsection{Currents}\label{subSection: Currents}
We denote with $\D^k(\Om)$ the vector space of compactly supported smooth $k$-differential forms. For a $k$-differential form $\om$ its comass is defined as
\begin{equation*}
\|\om\| =\sup \{\langle\om,\xi\rangle:\xi \mbox{ is a unit, simple $k$-vector}\}
\end{equation*}
Let $\D_k(\Om)$ be the dual to $\D^k(\Om)$ i.e. the space of \textit{$k$-currents} with its weak-$*$ topology. We denote with $\de$ the boundary operator that operates by duality as follows
\begin{equation*}
\langle \de \s,\om\rangle=\langle\s, \di \om\rangle\quad \mbox{ for all $(k-1)$-differential forms }\, \om.
\end{equation*}
The mass of a $k$-current $M(\s)$ is the supremum of $\langle \s,\om\rangle$ among all $k$-differential forms with comass bounded by $1$. 
For any $k$-current $\s$ such that both $\s$ and $\de \s$ are of finite mass we say that $\s$ is a normal $k$-current and we write $\s \in N_k(\Om)$. 
On the space $\D_k(\Om)$ we can define the \textit{flat norm} by
\begin{equation*}
\mathbb{F}(\s)=\inf\lt\{M(R)+M(S)\;:\; \s = R+\de S \mbox{ where } S\in \D_{k+1}(\Om) \mbox{ and } R\in \D_k(\Om)\rt\},
\end{equation*}
which metrizes the weak-$*$ topology on currents on compact subsets of $N_k(\Om)$.
 By the Radon-Nikodym theorem we can identify a $k$-current $\s$ with finite mass with the vector valued measure $\nu \mu_\s$ where $\mu_\s$ is a finite positive valued measure and $\nu$ is a $\mu_\s$-measurable map in the set of unitary $k$-vectors for the mass norm. In particular the action of $\s$ on $\om$ can be written as
 \begin{equation*}
 \langle\s,\om\rangle=\int_{\Om}\langle\om,\nu\rangle\di \mu_\s.
 \end{equation*}
 For a finite mass $k$-current the mass of $\s$ coincides with the total variation of the measure $\mu_\s$. 
 A $k$-current $\s$ is said to be $k$-rectifiable if we can associate to it a triplet $(\theta,\nu,\Sigma)$ such that 
  \begin{equation*}
 \langle\s,\om\rangle=\int_{\Sigma}\theta\langle\om,\nu\rangle\dH^k
 \end{equation*}
 where $\Sigma$ is a countably $k$-rectifiable subset of $\Om$, $\nu$ at $\H^k$ a.e. point is a unit simple $k$-vector that spans the tangent plane to $\Sigma$ and $\theta$ is an $ L^1(\Om,\H^k{\restr\Sigma})$ function that can be assumed positive. We will denote with $R_k(\Om)$ the space of these $k$-rectifiable currents. Among these we name out the subset $P_k(\Om)$ of $k$-rectifiable currents for which $\Sigma$ is a finite union of polyhedra, these will be called \textit{polyhedral chains}. Finally the \textit{flat chains} $F_k(\Om)$ consist of the closure of $P_k(\Om)$ in the weak-$*$ topology. By the scheme of Federer~\cite[4.1.24]{Federer} it holds
\begin{equation*}
P_k(\Om)\subset  N_k(\Om)\subset F_k(\Om).
\end{equation*}
\begin{remark}[$1$-Currents and Vector Measures]
Since the vector spaces $\Lambda_1 \R^n$, $\Lambda^1 \R^n$  identify with $ \R^n$, any vector measure $\s\in \M(\Om,\R^n)$ with finite mass indentifies with a $1$-current with finite mass and viceversa. The divergence operator acting on measures is defined by duality as the boundary operator for currents. In the following $\s\in\M(\Om,\R^n)$ is called a rectifiable vector measure if it is $1$-rectifiable as $1$-current. In the same fashion we define polyhedral $1$-measures.
\end{remark}

\subsection{Functionals defined on flat chains}\label{subSection: funcOnCurr}
For $f:\R\mapsto\R^+$ an even function we define a functional 
\begin{align*}
P_k(\Om)\qquad&\longto\qquad\R_+,\\
P= \sum_j (m_j,\nu_j,\Sigma_j)&\longmapsto \F(P)=\sum_jf(m_j)\H^k(\Sigma_j),
\end{align*}
on the space of polyhedral currents. Under the assumption that $f$ is lower semi-continuous and subadditive, $\F$ can be extended to a lower semi-continuous functional by relaxation
\begin{align*}
F_k(\Om)\qquad&\longto\qquad\;\;\qquad\R_+,\\
P\qquad\;\;\;&\longmapsto \;\;\F(P)=\inf \lt\{\liminf_{P_j\rw P}\F(P):(P_j)_j\subset P_k(\Om) \mbox{ and } P_j\rw P \rt\}.
\end{align*}
as shown in~\cite[Section~6]{White1}. Furthermore, in~\cite{CoDeMaSt17} the authors show  that if $f(t)/t\rw\oo$ as $t\rw0$, then $\F(\s)<\oo$ if and only if $\s$ is rectifiable and for any such $\s$ the functional takes the explicit form
\begin{equation}\label{Def: F}
 \F(\s)=\int_\Sigma f(m(x))\dH^k(x)\qquad \mbox{if } \s=(m,\nu,\Sigma).
\end{equation}
To conclude this subsection let us recall a sufficient condition for a flat chain to be rectifiable, proved by White in~\cite[Corollary 6.1]{White2}. 
\begin{theorem}\label{teo: krectifiable}
Let $\s\in F_k(\Om)$. If $M(\s)+M(\de \s)<\oo$ and if there exists a set $\Sigma\subset\Om$ with finite $k$-dimensional Hausdorff measure such that $\s=\s\restr\Sigma$ then $\s\in R_k(\Om)$ i.e., $\s$ writes as $(m,\nu,\Sigma)$.
\end{theorem}
In the context of vector measures the theorem writes as
\begin{theorem}\label{teo: 1rectifiable}
Let $\s\in\M(\Om,\R^n)$. If $|\s|(\Om)+|\nb \cdot s|(\Om)<\oo$, $\nb\cdot\s$ is at most a countable sum of Dirac masses and there exists $\Sigma$ with $\H^1(\Sigma)<\oo$ and $\sigma=\sigma \restr\Sigma$ then $\s$ is a rectifiable vector measure in the sense expressed in Subsection~\ref{subSection: Currents}.
\end{theorem}
\subsection{Reduced problem results in dimension $n-k$}\label{subSection: reducedProblem}
This subsection is devoted at introducing some notation and results corresponding to the case $k=0$. In the sequel, these results are used to describe the energetical behaviour of the  $(n-k)$-dimensional slices of the configuration $(\sigma\e,u\e)$. We postopone the proofs to Appendix~\ref{appendix: Reduced Problem}. We set $d=n-k$, $p > d$ and consider $\eps$ to be a sequence such that $\eps\dw0$. Let $B_r(0)\subset \R^d$ be the ball of radius $r$ centered in the origin, we consider the functional
\begin{equation*}\label{Def: Eeps}
  E_{\eps,a}(\vartheta,u;B_r):=\int_{B_{r}}\lt[\eps^{p-d}|\nb u|^p+\frac{(1-u)^2}{\eps^{d}}+\frac{u|\vartheta|^2}{\eps}\rt]\dx
\end{equation*} 
where $u\in W^{1,p}(B_r)$ is constrained to satisfy the lower bound $u\geq a\,\eps^{d+1}=:\eta$ 
 and $\vartheta\in L^2(B_r)$ is such that $\supp(\vartheta)\subset B_{\tilde r}$ with $0<\tilde r <r$, $\|\vartheta\|_1= m$. 
This leads to define the set
\begin{equation*}
Y_{\eps,a}(m,r,\tilde r)=\lt\{(\vartheta,u)\in L^2(B_r)\times W^{1,p}(B_r,[\eta,1])\,:\, \|\vartheta\|_1= m\mbox{ and } \supp(\vartheta)\subset B_{\tilde r}\rt\},                                                                                                
\end{equation*}
and the optimization problem
 \begin{equation}\label{def:min1}
 f_{\eps,a}^d(m,r,\tilde r)=\inf_{Y_{\eps,a}(m,r,\tilde r)}E_{\eps,a}(\vartheta,u;B_r).
 \end{equation}
 Let $f^d_a:[0,+\oo)\longto \R_+$ be defined as
\begin{equation}\label{Def: reducedfd}
 f^d_a(m)=\begin{dcases}
       \;\min_{\hat r>0}\quad\lt\{\frac{a \,m^2}{\;\om_{d}\;\hat r^{d}}+\;\om_{d}\;\hat r^{d}+(d-1)\;\om_{d}\;q^d_{\oo}(0,\hat r)\rt\},&\mbox{ for }m>0,\\
       \\
       \qquad\qquad\qquad0, &\mbox{ for }m=0,
      \end{dcases}
\end{equation}
with
\begin{equation}\label{eq: ginfinito}
q^d_{\oo}(\xi,\hat r):=\inf\lt\{\int_{\hat r}^{+\oo}t^{d-1}\lt[|v'|^p+(1-v)^2\rt]\dt\, :\, v(\hat r)=\xi \,\mbox{ and  } \,\lim_{t\rw+\oo}v(t)=1\rt\},
\end{equation}
for $\hat r>0$, $\xi\geq0$. We have the following results
\begin{proposition}\label{prop: structurefm}
For any $r>\tilde r>0$, it holds
\begin{equation}\label{eq: n-1liminf}
 \liminf_{\eps\dw0}f_{\eps,a}^d(m,r,\tilde r)\geq f^d_a(m).
\end{equation} 
There exists a uniform constant $\kappa:=\kappa(d,p)$ such that 
\begin{equation}\label{eq: reducedconstant}
 f^d_a(m)\geq \kappa\qquad \qquad \mbox{ for every } m>0.
\end{equation} 
\end{proposition}
\begin{proposition}\label{prop: flimsup}
For fixed $m>0$ let $r_*$ be the minimizing radius in the definition of $f^d_a(m)$~\eqref{Def: reducedfd}. For any $\d>0$ and $ \eps$ small enough there exist a function $\vartheta\e= c\mathbf{1}_{B_{r_*\eps}}$ with $c>0$  such that $\int_{B_r}\vartheta\e=m$ and a nondecreasing radial function $u\e:B_r\mapsto[\eta,1]$ such that $u\e(0)=\eta$, $u\e=1$ on $\de B_r$ and
\begin{equation}\label{eq: n-1recovery}
\quad E_{\eps,a}(\vartheta\e,u\e;B_r)\leq f^d_a(m)+\d.
\end{equation} 
\end{proposition}
\begin{proposition}\label{prop: fproprieties}
The function $f^d_a$ is continuous in $(0,+\infty)$, increasing, sub-additive and $f^d_a(0)=0$.
\end{proposition}

\section{The $1$-dimensional problem}\label{Section: 1Problem}
\subsection{Compactness}\label{subSection: 1Compactness}
We prove the compactness Theorem~\ref{teo: 1compactness} for the family of functionals $(\F_{\eps,a})\e$. Let us  consider a family of functions $(\s\e,u\e)_{\eps\dw0}$, such that $(\s\e,u\e)\in X\e(\Om)$ and 
\begin{equation}\label{energy bound}
\F_{\eps,a}(\s\e,u\e;\Om)\,\leq\, F_0.
\end{equation} 
As a first step we prove:
\begin{lemma}\label{lemma: mass bound}
Assume $a>0$. There exists $C\geq0$, only depending on $\Om$, $F_0$ and $a$ such that 
\begin{equation}
\label{uniform mass bound}
\int_{\Om} |\s\e |\,\leq\, C, \qquad \forall\,\eps.
\end{equation}  
As a consequence there exist a positive Radon measure $\mu\in(\R^n,\R_+)$ supported in $\ov\Om$ and a vectorial Radon measure $\s\in \M(\overline \Omega,\R^n)$ with $\nabla\cdot \s=\sum a_j\delta_{x_j}$ and $|\s|\leq\mu$  such that up to a subsequence
\begin{equation*}
u\e\to 1\mbox{ in }L^2(\Om),\qquad |\s\e|\, \stackrel{*}\rightharpoonup\, \mu \mbox{ in }\M(\R^n), \qquad \s\e\, \stackrel{*}\rightharpoonup\, \s \mbox{ in }\M(\R^n,\R^n).
\end{equation*}
\end{lemma}
\begin{proof}~We divide the proof into three steps.\\

\medskip
\noindent\textit{ Step 1.} We start by proving the uniform bound~\eqref{uniform mass bound}. Let $\lambda\in(0,1]$ and let 
\begin{equation*}
\Om_{\lambda}\,:=\, \lt\{x\in \Om\, :\, u\e>\lambda\rt\}.
\end{equation*}
Being $\s\e$ square integrable we identify the measure $\s\e$ with its density with respect to $\mathscr{L}^n$. Therefore splitting the total variation of $\s\e$, we write
\begin{equation*}
 |\s\e|(\Om)=\int_{\Om}|\s\e|\dx=\int_{\Om_\lambda}|\s\e|\dx+\int_{\Om\setminus\Om_{\lambda}}|\s\e|\dx.
 \end{equation*} 
 We estimate each term separately. By Cauchy-Schwarz inequality we have
\begin{equation*}
\int_{\Om_\lambda} |\s\e|\, \leq\, \lt(\int_{\Om_\lambda} \dfrac{u\e|\s\e|^2}{\eps}\rt)^{1/2}\lt(\int_{\Om_\lambda} \dfrac{\eps}{u\e}\rt)^{1/2}.
\end{equation*}
Since $\lambda<u\e<1$ on $\Om_\lambda$ and $\int_{\Om_\lambda} (u\e|\s\e|^2)/(\eps)\dx$ being bounded by $\F_{\eps,a}(\s\e,u\e)$ from the previous we get
\begin{equation*}
\int_{\Om_\lambda} |\s\e|\, \leq\, \lt(\int_{\Om_\lambda} \dfrac{u\e|\s\e|^2}{\eps}\rt)^{1/2}\sqrt{\frac{|\Om|\eps}{\lambda}} \, \leq\, \sqrt{\frac{|\Om|\;\eps\;F_0}{\lambda}}\;.
\end{equation*}
Next, in $\Om\sm\Om_\lambda$, by Young inequality, we have 
\begin{equation*}
2\;\int_{\Om\sm\Om_\lambda} |\s\e|
\, \leq\, 
	\int_{\Om\sm\Om_\lambda}\dfrac{u\e|\s\e|^2}{\eps}
	+ \int_{\Om\sm\Om_\lambda} \dfrac{\eps}{u\e}.
	\end{equation*}
Using $u\e\geq\eta(\eps)$, $\eta/\eps^{n}= a$ and $(1-\lambda)^2\leq (1-u\e)^2$ in $\Om\sm\Om_\lambda$, we obtain 
\begin{equation*}
\int_{\Om\sm\Om_\lambda} |\s\e|
\, \leq\, \frac{1}{2}\int_{\Om}\dfrac{u\e|\s\e|^2}{\eps}	+ \dfrac{\eps^{n}}{2\,\eta\;(1-\lambda)^2} \int_{\Om} \dfrac{(1-u\e)^2}{\eps^{n-1}}\leq \frac{F_0}{2}+\frac{F_0}{2\,a\,(1-\lambda)^2}.
	\end{equation*}
Hence
\begin{equation*}
 |\s\e|(\Om)\leq \frac{F_0}{2}\,+\,\frac{F_0}{2\;a\;(1-\lambda)^2}\,+\,\;\sqrt{\frac{|\Om|\;\eps\;F_0}{\lambda}}.
 \end{equation*} 
As $a>0$, this yields~\eqref{uniform mass bound}. \medskip

\noindent\textit{ Step 2.} We easily see from $\int_{\Om}(1-u\e)^2\,\leq\, F_0\eps^{n-1}$ that $u\e\to 1$ in $L^2(\Om)$ as $\eps\dw0$.
\medskip

\noindent\textit{ Step 3.} The existence of the Radon measures $\mu$ and $\s$ such that, up to extraction,  $|\s\e|\, \stackrel{*}\rightharpoonup\, \mu$ and $\s\e\, \stackrel{*}\rightharpoonup\, \s$ follows from~\eqref{uniform mass bound}. The properties on the support of $\mu$, on the divergence  of $\s$ and  the fact that $|\s|\leq\mu$ follow from the respective properties of $\s\e$.    
\end{proof}

We have just showed that the limit $\s$ of a family $(\s\e,u\e)\e$ equibounded in energy is bounded in mass. In what follows, we assume $a\geq0$ and that $\s\e$ is bounded in mass. We show that the limiting $\s$ is rectifiable.
\begin{proposition}
\label{proposition: length bound}
Assume $a\geq0$ and that the conclusions of Lemma~\ref{lemma: mass bound} hold true. There exists a Borel subset $\Sigma$ with finite length and a Borel measurable function $\nu\,:\,\Sigma\to S^{n-1}$ such that $\sigma = \nu |\sigma|\restr \Sigma$. Moreover, we have the following estimate,
\begin{equation*}
\H^1(\Sigma)\,\leq\, C_* F_0,
\end{equation*}
where the constant $C_* \geq 0$ only depends on $d$ and $p$. 
\end{proposition}
\noindent This proposition together with Lemma~\ref{uniform mass bound} and Theorem~\ref{teo: 1rectifiable} leads to 
\begin{proposition}\label{prop: sigma1rectifiable}
$\s$ is a $1$-rectifiable vector measure and in particular $\Sigma $ is a countably $\H^1$-rectifiable set.
\end{proposition}
\noindent The latter ensures that the limit couple $(\s,1)$ belongs to $X$ and concludes the proof of Theorem~\ref{teo: 1compactness}. We now establish Proposition~\ref{proposition: length bound}

\medskip
\textit{Sketch of the proof:} We first define $\Sigma$. Then we show in Lemma~\ref{lemma: local length  Sigma}  that for $x\in \Sigma$, we have $\liminf_{\eps\dw 0}\F_{\eps,a}(\s\e,u\e;B(x,r_j))\geq \kappa r_j$ for a sequence of radii $r_j\dw0$ and $\kappa>0$. The proof of the lemma is based on slicing and on the results of Appendix~\ref{appendix: Reduced Problem}. The proposition then follows from an application of the Besicovitch covering theorem.

\medskip
First we introduce the Borel set 
\begin{equation*}
\wtilde \Sigma\, :=\, \left\{ x\in \Om\,:\,\forall r>0,\ |\sigma|(B_r(x))>0\mbox{ and }\exists \nu=\nu(x)\in \S^{n-1}\, \mbox{such that } \nu=\lim_{r\dw 0}\frac{\sigma(B_r(x))}{|\sigma|(B_r(x))}\right\}.
\end{equation*}
We observe that by Besicovitch derivation theorem,
\begin{equation*}
\sigma = \nu |\sigma|\restr \wtilde \Sigma.
\end{equation*}
Next we fix $\theta\in (0,1/4^n)$ and define
\begin{equation*}
\Gamma\,:=\, \lt\{ x\in \wtilde \Sigma\, :\, \exists\,r_0>0\mbox{ such that} ,\ \dfrac{|\sigma|(B_{r/4}(x))}{|\sigma|(B_r(x))}\leq \theta \ \mbox{ for every } r\in (0,r_0] \rt\}.
\end{equation*}
We show that this set is $|\sigma|$-negligible.
\begin{lemma}\label{lemma: Gamma negligible}
We have $|\s|(\Gamma)=0$.
\end{lemma}
\begin{proof}
Let $x\in\Gamma$. Applying the inequality $|\sigma|(B_{r/4}(x))\leq \theta |\sigma|(B_r(x))$ with $r=r_k=4^{-k}r_0$, $k\geq 0$, we  get 
$|\sigma|(B_{r_k})\, \leq\, \theta^k|\sigma|(B_{r_0})$. Hence there exists $C\geq 0$ such that
\begin{equation*}
|\sigma|(B_r(x))\, \leq\, C r^{(\ln 1/\theta)/(\ln 4)}.
\end{equation*}
Noting, $\lambda=(\ln \frac{1}{\theta})/(\ln 4)$, we have by assumption $\lambda>n$. Therefore, for every $\xi>0$ there exists $r_\xi=r_\xi(x)\in (0,1)$ such that 
\begin{equation*}
|\sigma |(B_{r_\xi}(x))\, \leq\, \xi |B_{r_\xi}(x)| .
\end{equation*}
Now, for $R>0$, we cover $\Gamma\cap B_R$ with balls of the form $B_{r_\xi(x)}(x)$. Using Besicovitch covering theorem, we have 
\begin{equation*}
\Gamma\cap B_R\, \subset\, \cup_{j=1}^{N(n)} \mathcal{B}_j 
\end{equation*}
where $N(n)$ only depends on $n$ and each ${\cal B}_j$ is a (finite or countable) disjoint union of balls of the form $B_{r_\xi(x_k)}(x_k)$. Then we get
\begin{equation*}
|\sigma|(\Gamma\cap B_R)\, \leq\, \sum_{j=1}^{N(n)}|\sigma|({\cal B}_j)\, \leq\, N(n) \xi |{\cal B}_j|\, \leq\, N(n)|B_{R+1}|\xi.
\end{equation*}
Sending $\xi$ to 0 and then $R$ to $\infty$, we obtain $|\sigma|(\Gamma)=0$.
\end{proof}
Set $\Sigma\,:=\, \wtilde \Sigma\setminus \Gamma$, from Lemma~\ref{lemma: Gamma negligible}, we have $\sigma = \nu |\sigma|\restr \Sigma. $ Recall that $\mathscr{S}=\{x_1,\cdots,x_{n_P}\}$.
\begin{lemma}\label{lemma: local length Sigma}
For every $x\in \Sigma\setminus \mathscr{S}$, there exists a sequence $(r_j)=(r_j(x))\subset (0,1)$ with $r_j\dw0$ such that
\begin{equation*}
\liminf_{\eps\dw0} \F_{\eps,a}(\s\e,u\e;B(x,r_j))\, \geq\, \sqrt{2}\,\kappa\; r_j,
\end{equation*}
where $\kappa$ is the constant of Proposition~\ref{prop: structurefm}.
\end{lemma}
\begin{proof}
Let $x\in \Sigma\setminus \mathscr{S}$. Without loss of generality, we assume $x=0$ and $\nu(x)=e_1$. Let $\xi>0$ be a small parameter to be fixed later. From the definition of $\Sigma$, there exists a sequence $(r_j)=(r_j(x))\subset (0,d(x,\mathscr{S}))$ such that for every $j\geq0$,
\begin{equation}\label{pr. local length 1}
\sigma(B_{r_j})\cdot e_1\, \geq\, (1-\xi) |\sigma|(B_{r_j})\qquad\mbox{and}\qquad |\sigma|(B_{r_j/4})\, \geq\, \theta |\sigma|(B_{r_j}). 
\end{equation}
Let us fix $j\geq0$ and set, to simplify the notation, $r=r_j$ and $r_*=r/\sqrt{2}$. Recall the notation $x=(x_1,x')\in\R\times\R^{n-1}$ and define the cylinder
\begin{equation*}
C_{r_*}:=\lt\{x: |x_1|\leq r_*\quad\mbox{ and }\quad |x'|\leq r_*\rt\}
\end{equation*}
 so that $C_{r_*}\subset B_{r}$  and $B_{r/4}\subset C_{{r_*}/2}$, as shown in figure~\ref{figura1}. Let $\chi\in C^\oo_c (\R^{n-1},[0,1])$ be a radial cut-off function such that $\chi(x')=1$ if $|x'|\leq \frac{1}{2}$ and $\chi(x')=0$ for $|x'|\geq\frac{3}{4}$. Then, we note $\chi_{r_*}(x')=\chi(x'/{r_*})$ and for $s\in [-r,r]$, we set
\begin{equation*}
\forall s\in [-r,r],\qquad g\e(s)\,:=\, e_1\cdot \int_{B_{r_*}'} \s\e(s,x')\,\chi_{r_*}(x')\, \dx'.
\end{equation*}
  \begin{figure}[!ht]
  \centering
   \includegraphics[scale=0.7]{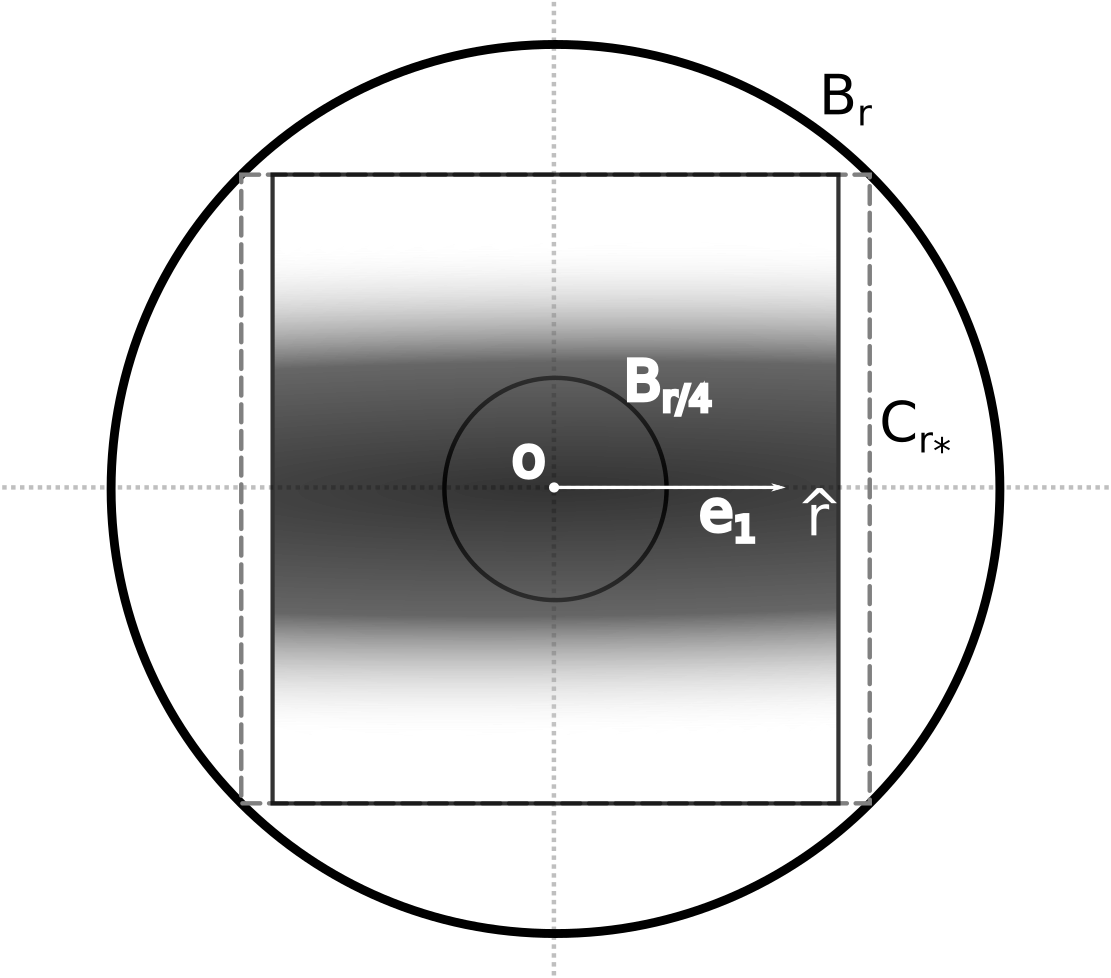}
   \caption{Illustration of the sections of $B_r$, $B_{r/4}$ and $C_{r_*}$. In grayscale we represent the level sets of the function $\chi_{r_*}(x')\mathbf{1}_{[-\hat r,\hat r]}$.}\label{figura1}
  \end{figure}
Since $\s\e$ is divergence free, $e_1\cdot\s\e(\cdot,s)$ has a meaning on the hyperplane $\{x_1=s\}$ in the sense of trace, moreover, $g\e$ is continuous. 
Now, let us fix $\hat r\in [(1-\xi){r_*},{r_*}]$ such that $\mu(\{-\hat r,\hat r\}\times B_r')=0$ (which holds true for a.e. $\hat r \in [(1-\xi){r_*},{r_*}]$) and let us define the mean value,
\begin{equation*}
\bar g\e\,:=\, \dfrac1{2\hat r}\int_{-\hat r}^{\hat r} g\e(s)\,\di s.
\end{equation*}
From $\s\e\,\stackrel*\rightharpoonup\,\s$, $|\s\e|\,\stackrel*\rightharpoonup\,\mu$, we have
\begin{equation}\label{pr. local length 2}
\lim_{\eps\dw0} \bar g\e\, =\, \lt(\dfrac1{2\hat r}\int_{(-\hat r,\hat r)\times B'_{r_*}}\chi_{r_*}(x')\,\di\s(s,x')\rt)\cdot e_1  \, =:\, \ov m.                                                                                                                                               \end{equation} 
From~\eqref{pr. local length 1}, we see that $\ov m>0$ for $\xi$ small enough. Indeed, we have 
\begin{eqnarray*}
(1-\xi)|\s|(B_{r})\, \leq\,+ 2\hat r \ov m+  \s (B_{r})\cdot e_1& =& \int_{B_{r}}\lt(1-\chi_{r_*}(x')\mathbf{1}_{[-\hat r,\hat r]}\rt)\, \ds(s,x')\cdot e_1\\
& \leq  & 2\hat r \ov m + \int_{B_{r}}\lt(1-\chi_{r_*}(x')\mathbf{1}_{[-\hat r,\hat r]}\rt)\, \di|\s|(s,x')\\
& \leq  & 2\hat r \ov m+|\s|(B_{r})- \int_{B_{r}}\chi_{r_*}(x')\mathbf{1}_{[-\hat r,\hat r]}\, \di|\s|(s,x').
\end{eqnarray*}
Since by construction $\chi_{r_*}(x')\mathbf{1}_{[-\hat r,\hat r]}\geq\mathbf{1}_{B_{r/4}}$, using the second inequality of~\eqref{pr. local length 1}, we have
\begin{equation*}
\ov m\, \geq\, \dfrac1{2\hat r}(\theta-\xi)|\s|(B_{r})\,>\,0,
\end{equation*} 
for $\xi$ small enough. Similarly, denoting $\Pi \, :\, \R^{n}\to\R^{n-1},\ (t,x')\mapsto x'$ the orthogonal projection onto the last $(n-1)$ coordinates, we deduce again from~\eqref{pr. local length 1} that 
\begin{equation}\label{pr. local length 3}
|\Pi \sigma|(C_{r_*}) \leq \frac{\sqrt{\xi} \ov m}{\t -\xi} 2 \hat r.
\end{equation}
Now, for $\eps$ small enough, we have $\nb\cdot\sigma\e=0$ in $C_{r_*}$. Using this, we have for almost every $s,t \in [-\hat r,\hat r]$,  with $s<t$,
\begin{equation*}
g\e(t)-g\e(s)\, =\, \int_s^t\lt[\int_{B'_{r_*}} \s\e(x',h)\cdot \nb'\chi_{r_*}(x')\, \dx'\rt]\, \di h.
\end{equation*}
Integrating in $s$ over $(-\hat r,\hat r)$, we get for almost every $t\in [-r,r]$, 
\begin{equation*}
g\e(t)-\bar g\e\, =\, \dfrac1{2\hat r}\int_{ (-\hat r,
\hat r)\times B'_{r_*}} \phi_t(x',h)\cdot \sigma\e(x',h)\, \dx'\, \di h
\end{equation*}
with
\begin{equation*}
\phi_t(h,x')\, =\, 
     \begin{dcases}
   	        \lt(h+\hat r\rt)\nb'\chi_{r_*}(x') & \mbox{if } h<t,\\
	         \lt(h-\hat r\rt)\nb'\chi_{r_*}(x') & \mbox{if } h>t.
     \end{dcases}
     \end{equation*}
We deduce the following convergence 
\begin{equation}
g\e(t)-\ov m\stackrel{\eps\dw0}\longto\,   \dfrac1{2\hat r}\int_{ (-\hat r,
\hat r)\times B'_{r_*}} \phi_t(h,x')\cdot \,\di\s(h,x').
\end{equation}
in the $L^1(-\hat r,\hat r)$ topology. Using~\eqref{pr. local length 3}, we see that the above right hand side is bounded by $c\frac{\sqrt{\xi}}{\t-\xi} \ov m$. Taking into account~\eqref{pr. local length 3} and the continuity of $g\e$, we conclude that  
\begin{equation*}
\liminf_{\eps\dw0} g\e(t)\geq \lt(1-c\frac{\sqrt{\xi}}{\t-\xi}\rt) \ov m \quad \mbox{ for } \quad t\in [-\hat r,\hat r].
\end{equation*}
Next, by decomposing the integral we have 
\begin{equation}
\begin{aligned}\label{pr. local length 5}
\F_{\eps,a}(\s\e,u\e;B_{r})\, &\geq\, \int_{-\hat r}^{\hat r}\int_{B'_{r_*}}\lt[\eps^{p-n+1}|\nb u\e|^p+\frac{(1-u\e)^2}{\eps^{n-1}}+\frac{u\e|\s\e|^2}{\eps}\rt]\dx'\, \dt\\
&\geq\,\int_{-\hat r}^{\hat r}\int_{B'_{r_*}}\lt[\eps^{p-n+1}|\nb u\e|^p+\frac{(1-u\e)^2}{\eps^{n-1}}+\frac{u\e|\chi_{r_*}(x')\s\e|^2}{\eps}\rt]\dx'\, \dt.
\end{aligned}
\end{equation}
Let us set  
\begin{equation*}
\vartheta^t\e(x'):=|\chi_{r_*}(x')\s\e(t,x')|.
 \end{equation*}
By construction $\vartheta^t\e$ has the properties:
\begin{itemize}
\item $\vartheta^t\e\in L^1(B'_{r_*})$,
\item $\liminf_{\eps\dw0}\int_{B'_{r_*}}\vartheta^t\e(x')\, dx'\,\geq \, \liminf_{\eps\dw0} g\e(t)\, \geq\, \lt(1-c\frac{\sqrt{\xi}}{\t-\xi} \rt)\ov m=\tilde m>0$,
\item $\supp(\vartheta^t\e)\subset B'_{\tilde r}$ with $\tilde r :=\frac{3}{4}{r_*}<{r_*}$.
\end{itemize}
By definition of the minimization problem introduced in Subsection~\ref{subSection: reducedProblem}, we have
\begin{equation}\label{pr. local length 6}
\F_{\eps,a}(\s\e,u\e;B_{r})\, \geq\,\int_{-\hat r}^{\hat r}\lt[\inf_{(\vartheta,u)\in Y_{\eps,a}(\tilde m,r,\tilde r)} E_{\eps,a}(\vartheta,u;B_r)\rt]\, \dt=\,\int_{-\hat r}^{\hat r}f^{\tilde r}\e\lt(\tilde m\rt)\, \dt.
\end{equation}
Taking the infimum limit, by Fatou's lemma and equation~\eqref{eq: reducedconstant} of Proposition~\ref{prop: structurefm} we get
\begin{equation*}
\liminf_{\eps\dw0} \F_{\eps,a}(\s\e,u\e;B_{r})\geq \int_{-\hat r}^{\hat r}\liminf_{\eps\dw 0}f^{\tilde r}\e\lt(\tilde m\rt)\, \dt\,\geq\, 2\,\hat r\, \kappa.
\end{equation*}
The latter holds for almost every  $\hat r \in [(1-\xi)r_*,r_*]$ and eventually, since the $r_*=r/\sqrt{2}$, we conclude
\begin{equation*}
\liminf_{\eps\dw0} \F_{\eps,a}(\s\e,u\e;B_{r})\geq \, \sqrt{2}\,  \kappa\,r.
\end{equation*}
 \end{proof}
 The proof of Proposition~\ref{proposition: length bound} is then obtained via the Besicovitch covering theorem~\cite{MR3409135}.
 

\subsection{$\Gamma$-liminf inequality}\label{subSection: 1liminf}
In this subsection we prove the $\Gamma-\liminf$ inequality stated in Theorem~\ref{teo: 1Gammaliminf}.
\begin{proof}[Proof of Theorem~\ref{teo: 1Gammaliminf}.]
 With no loss of generality we assume that $\liminf_{\eps\dw0} \F_{\eps,a}(\s\e,u\e)<+\oo$ otherwise the inequality is trivial. For a Borel set $A\subset \Omega$, we define
 \[ 
 H(A)\,:=\,\liminf_{\eps\dw0}\F_{\eps,a}(\s\e,u\e;A),
 \] 
so that $H$ is a subadditive set function. By assumption, the limit measure $\sigma$ is $1$-rectifiable; we write $\s= m\,\nu\, \H^1\restr \Sigma$. Furthermore we can assume $\sigma$ to be compactly supported in $\Omega$. Consider a convex open set $\Omega_0$ such that $\supp(\nabla\cdot\s)=\mathscr{S}\subset\subset\Omega_0\subset\subset \Omega$ and let $h:=[0,1]\times\R^n\to \R^n$ be a smooth homotopy of the indentity map on $\R^n$ onto a contraction of  $\overline\Omega$ into $\overline \Omega_0$ such that $h(t,\cdot)$ restricted to $\Omega_0$ is the identity map, for any $t\in[0,1] $. Let $\s_t=h(t,\cdot)\sharp\s$, indeed $\liminf_{t\dw0}\F(\s_t,1)\geq \F(\s,1)$ as  $\s_t\rwstar\s$. Further $\nabla\cdot \sigma_t=\nabla \cdot \s$ since $h(t,\cdot)$ is the identity on $\mathscr{S}$.
Now we claim that 
\begin{equation}\label{eq: liminflocale}
\liminf_{r\dw0}\dfrac{H\lt(\ov{B(x,r)}\rt)}{2r} \, \geq\,f_a(m(x))\qquad\mbox{for $\H^1$-almost every $x\in\Sigma$.}
 \end{equation}
Let us fix $\lambda\geq 1$ and let us note $f_{a,\lambda}(t):=\min(f_a(t),\lambda)$. We then introduce the Radon measure 
\[
H_\lambda'(A):= \int_{\Sigma\cap A} f_{a,\lambda}(m)\, \dH^1.
\]
Now, let $\delta\in(0,1)$. Assuming that~\eqref{eq: liminflocale} holds true, there exists $\Sigma'\subset\Sigma$ with $\H^1(\Sigma\backslash\Sigma_0)=0$ such that for every $x\in\Sigma_0$, there exists $r_0(x)>0$ with 
\[
(1+\delta) H\lt(\ov{B(x,r)}\rt)\,\geq\, 2r f_{a,\lambda}(m(x)) \qquad\mbox{ for every $r\in(0,r_0(x))$}.
\]
By the Besicovitch differentiation Theorem, there exists $\Sigma_1\subset\Sigma$ with $\H^1(\Sigma\backslash\Sigma_1)=0$ such that for every $x\in\Sigma_1$, there exists $r_1(x)>0$ with 
\[
(1+\delta) 2r f_a(m(x))\,\geq\,H_\lambda'\lt(\ov{B(x,r)}\rt)\qquad\mbox{ for every $r\in(0,r_1(x))$}.
\]
We consider the familly $\mathcal{B}$ of closed balls $\ov{B(x,r)}$ with $x\in \Sigma_0\cap \Sigma_1$ and $0<r<\min(r_0(x),r_1(x))$ and we apply the Vitali-Besicovitch covering theorem~\cite[Theorem 2.19.]{Am_Fu_Pal} to the family $\mathcal{B}$ and to the Radon measure $H'_\lambda$. We obtain a disjoint family of closed balls $\mathcal{B}'\subset\mathcal{B}$ such that  
\[
H'_\lambda(\Omega)=H'_\lambda(\Sigma)=\sum_{\ov{B(x,r)}\in \mathcal{B}'}H'_\lambda\lt(\ov{B(x,r)}\rt) \,\leq\,  (1+\delta)^2  \sum_{\ov{B(x,r)}\in \mathcal{B}'} H\lt(\ov{B(x,r)}\rt)\,\leq \,  (1+\delta)^2  H(\Omega).
\]
Sending $\lambda$ to infinity and then $\delta$ to 0, we get the lower bound $H(\Omega)\geq \int_\Sigma f_a(m)\, d\mathcal{H}^1$ which proves the theorem.\medskip

Let us now establish the claim~\eqref{eq: liminflocale}. Since $\s$ is a rectifiable measure, we have for $\mathcal{H}^1$-almost every $x\in\Sigma$,
\begin{equation} 
\label{eq: liminf1}
\dfrac1{2r}\int \vhi(x+ry)\, \di|\sigma|(y)\ \stackrel{r\dw0}\longrightarrow\ m(x)\,\int_\R\vhi(t\nu(x))\, \di t\qquad\mbox{for every $\vhi\in C_c(\R^n)$},
\end{equation}
and
\begin{equation}
\label{eq: liminf2}
\dfrac1{2r}\int_{B(x,r)\cap\Sigma} |\nu(y)-\nu(x)|\, \di |\sigma|(y)\ \stackrel{r\dw0}\longrightarrow\ 0.
\end{equation}
Let $x\in \Sigma\setminus \mathcal{S}$ be such a point. Without loss of generality, we assume $x=0$, $\nu(0)=e_1$ and $\ov{m}:=m(0)>0$. Let $\delta\in(0,1)$. Our goal is to establish a precise lower bound for $\F_{\eps,a}(\s\e,u\e;C)$ where $C$ is a cylinder of the form
\[
C_{r}^\delta\,:=\, \lt\{x\in \R^n\, :\, |x_1|<\delta r,\, |x'|<r\rt\}.
\]
For this we proceed as in the proof of Lemma~\ref{lemma: local length Sigma}, here, the rectifiability of $\sigma$ simplifies the argument. Let $\chi^\delta\in C^\oo_c(\R^{n-1},[0,1])$ be a radial cut-off function with $\chi^\delta(x')=1$ if $|x'|\leq\delta/2$, $\chi^\delta(x')=0$ if $|x'|\geq \delta$. For $\eps>0$ and $r\in (0,d(0,\pt\Om))$, we define for $s\in(-r,r)$,
\[
 g^{\delta,r}\e(s)\,:=\, e_1\cdot \int_{\R^{n-1}} \s\e(s,x')\,\chi^\delta(x'/r)\, \dx'.
\]
We also introduce the mean value
\[
\ov{g^{\delta,r}\e}\, :=\, \dfrac1{2r}\int_{-r}^r g^{\delta,r}\e(s)\, \di s.
\]
From~\eqref{eq: liminf1}, we have for $r>0$ small enough, 
\[
\ov{g^{\delta,r}_0}\, :=\,  \dfrac1{2r}\int_{-r}^r  e_1\cdot \int_{\R^{n-1}} \s\e(s,x')\,\chi^\delta(x'/r)\, \dx\, \di s \, \geq\, (1-\delta)\ov m.
\]
For such $r>0$, we deduce from $\s\e\stackrel{*}\rightharpoonup\s$ that for $\eps>0$ small enough
\begin{equation}
\label{eq: liminf3}
\ov{g^{\delta,r}\e}\, :=\, \dfrac1{2r}\int_{-r}^r g^{\delta,r}\e(s)\, \di s\, \geq\, (1-2\delta)\ov m.
\end{equation}
We study the variation of $g^{\delta,r}\e(s)$. Using $\nb\cdot\s\e=0$ in $C_r^\delta$, we compute as in the proof of Lemma~\ref{lemma: local length Sigma},
\begin{equation*}
g^{\delta,r}\e(t)-\ov {g^{\delta,r}\e}\, =\, \dfrac1{2r}\int_{ (-r,
r)\times B_{\delta r}} \phi_t(x',h)\cdot \sigma\e(x',h)\, \dx'\, \di h
\end{equation*}
with
\begin{equation*}
\phi_t(h,x')\, =\, 
     \begin{dcases}
   	        \lt(h+\hat r\rt)\nb'\chi^\delta(x'/r) & \mbox{if } h<t,\\
	         \lt(h-\hat r\rt)\nb'\chi^\delta(x'/r) & \mbox{if } h>t.
     \end{dcases}
     \end{equation*}
Using again the convergence  $\s\e\stackrel{*}\rightharpoonup\s$, we deduce 
\[
g^{\delta,r}\e(t)-\ov {g^{\delta,r}\e}\ \stackrel{\eps\dw0}\longto\   \dfrac1{2r}\int_{ (-r,
r)\times B_{\delta r}} \phi_t(x',h)\cdot \di\sigma(x',h),
\]
in $L^1(-r,r)$. Now, since $e_1\cdot \nb'\chi^\delta\equiv 0$, we deduce from~\eqref{eq: liminf2} that the right hand side goes to 0 as $r\dw0$. Hence, for $r>0$ small enough,
\[
\lt| \dfrac1{2r}\int_{ (-r,r)\times B_{\delta r}} \phi_t(x',h)\cdot \sigma(x',h)\, \dx'\, \di h\rt|\, \leq\, \delta \ov m.
\]
Using~\eqref{eq: liminf3}, we conclude that for $r>0$ small enough and then for $\eps>0$ small enough, we have 
\[
g^{\delta,r}\e(t)\, \geq \, (1-3\delta)\ov m,\qquad \mbox{ for a.e. }t\in (-r,r).
\]
By definition of the codimension-0 problem, we conclude that 
\[
\F_{\eps,a}(\s\e,u\e;C_{r}^\delta)\ \geq\, 2r f^{n-1}_{\eps,a}\lt( (1-3\delta)\ov m\rt).
\]
Sending $\eps \dw 0$, we obtain
\[
H(C_{r}^\delta)\ \geq\, 2r f^{n-1}\lt( (1-3\delta)\ov m\rt).
\]
We notice that $H(B_{\sqrt{1+\delta^2}\,r})\geq H(C_{r}^\delta)$. Dividing by $2\sqrt{1+\delta^2}\,r$ and taking the liminf as $r\dw0$, we get 
\[
\liminf_{r\dw0}\dfrac{H(B_{\sqrt{1+\delta^2}\,r})}{2\sqrt{1+\delta^2}\,r} \, \geq \, \dfrac{f_a\lt( (1-3\delta)\ov m\rt)}{\sqrt{1+\delta^2}}.
\]
Sending $\delta$ to $0$, we get~\eqref{eq: liminflocale} by lower semi-continuity of~$f_a$.
\end{proof}

\subsection{$\Gamma$-limsup inequality}\label{subSection: 1limsup}
\begin{proof}[Proof of Theorem~\ref{teo: 1Gammalimsup}]~\\
Let us suppose $\F(\s,u;\ov\Om)<+\oo$, so that in particular $u\equiv 1$.  From Xia~\cite{Xia1}, we can assume $\s$ to be supported on a finite union of compact segments and to have constant multiplicity on each of them, namely polyhedral vector measures are dense in energy. We first construct a recovery sequence for a measure $\s$ concentrated on a segment with constant multiplicity. Then we show how to deal with the case of a polyhedral vector measures.\\

 \textit{Step 1. ($\s$ concentrated on a segment.) } Assume that $\s$ is supported on the segment $I=[0,L]\times\{ 0\}$ and writes as $m\cdot e_1\H^1\restr_{I}$. Consider $m$ constant so that $\nb\cdot \s=m(\d_{(0, 0)}-\d_{(L, 0) })$ and 
\begin{equation*}
\F(\s,1;\Om)=f_a(m)\,\H^1(I)=L\,f_a(m).
\end{equation*}
For $\d>0$ fixed, we consider the profiles
\begin{equation*}
\ov u\e(t):=
\begin{dcases}
\eta,&\mbox{ for } 0\leq t\leq  r_*\eps,\\
v_\d\lt(\frac{t}{\eps}\rt),&\mbox{ for } r_*\eps\leq  t\leq r,\\
1&\mbox{ for } r\leq  t,
\end{dcases}
\qquad\mbox{and}\qquad\vartheta\e=\frac{m\;\chi_{B'_{r_*\eps}}(x')}{\;\om_{n-1}\;   (\eps r_*)^{n-1}}
\end{equation*}
with $r_*$ and $ v_\d$, defined in Proposition~\ref{prop: flimsup} with $d=n-1$. Assume $r_*\geq1$ and let $d(x,I)$ be the distance function from the segment $I$ and introduce the sets
\begin{equation*}
I_{ r_*\eps}:=\lt\{x\in \Om\;:\;d(x,I)\leq r_* \eps\rt\},\qquad \mbox{ and }\qquad I_{ r}:=\lt\{x\in \Om\;:\;d(x,I)\leq r\rt\}.
\end{equation*}
Set $u\e(x)=\ov u\e(d(x,I))$ and $\ov\s^1\e=m\H^1\restr I*\rho\e$, where $\rho\e$ is the mollifier of equation~\eqref{point sources eps}. We first construct the vector measures 
\begin{equation*}
   \s\e^1=\ov\s^1\e \,e_1\qquad \mbox{and}\qquad \s\e^2(x_1,x')=\vartheta\e(|x'|)\; e_1.
  \end{equation*}
Alternatively, $\s^2\e=\s*\tilde\rho\e$  for the choice $\tilde\rho\e(x_1,x')=\chi_{B'_{r_*\eps}}(x')/\;\om_{n-1} (\eps r_*)^{n-1}$. Let us highlight some properties of $\s\e^1$ and $\s\e^2$. Both vector measures are radial in $x'$, with an abuse of notation we denote $\ov \s^1\e(x_1,s)=\ov \s^1\e(x_1,|x'|)$. Since, both  $\s\e^1$ and $\s\e^2$ are obtained trough convolution it holds  $\supp(\s\e^1)\cup\supp(\s\e^2)\subset I_{r_*\eps}$ and they are oriented by the vector $e_1$ therefore $|\s^1\e|=\ov\s^1\e$ and $|\s^2\e|=\vartheta\e$. Furthermore for any $x_1$, it holds
\begin{equation}\label{eq: corresponding normal flux}
\int_{\{x_1\}\times B'_{r_*\eps}}  \lt[\ov\s^1\e(x_1,x')-\vartheta\e(x')\rt]\dx'=0
\end{equation}
We construct $\s\e$ by interpolating between $\s^1\e$ and $\s^2\e$. To this aim consider a cutoff function $\zeta\e:\R\rw\R_+$ satisfying

 \begin{minipage}{.35\textwidth}
   \begin{gather*}
\qquad  \zeta\e(t)=1\qquad \mbox{ for } t\leq  r_* \eps\mbox{ or } t\geq L- r_* \eps,\\
\qquad  \zeta\e(t)=0\qquad \mbox{ for } 2\; r_* \eps\leq t\leq L-2\; r_* \eps,\\
\end{gather*} 
 \end{minipage}
  \begin{minipage}{.60\textwidth}
 \begin{equation*}
 \mbox{ and }\qquad \qquad\lt|\zeta\e'\rt|\leq \frac{1}{ r_* \eps}.
 \end{equation*}
 \end{minipage} ~\\
and set
 \begin{equation*}
 \begin{dcases}
   \s\e^3\cdot e_1=0,\\
 \s\e^3\cdot e_i(x_1,x')=-\zeta\e'(x_1)\;\frac{x_i}{|x'|^{n-1}}\;\int_{0}^{|x'|}s^{n-2}\lt[\ov\s\e^1(x_1,s)-\vartheta\e(s)\rt]\di s, &\mbox{ for } i=2,\dots, n.
 \end{dcases}
\end{equation*}
The integral corresponds to the difference of the fluxes of $\s^1\e$ and $\s^2\e$ through the $(n-1)$-dimensional disk $\{x_1\}\times B'$. For $\s^3\e$ we have the following \begin{multline}\label{divesigma3}
  \nb\cdot \s^3\e= -\zeta\e'(x_1)\sum_{i=2}^n\lt[\lt(\frac{1}{|x'|^{n-1}}-\frac{(n-1)x_i^2}{|x'|^{n+1}}\rt)\int_{0}^{|x'|}s^{n-2}\lt[\ov\s\e^1(x_1,s)-\vartheta\e(s)\rt]\di s\rt.\\+\lt.\frac{x_i^2}{|x'|^2}\lt[\ov\s\e^1(x_1,|x'|)-\vartheta\e(|x'|)\rt]\rt]=-\zeta\e'(x_1)\lt[\ov\s\e^1(x_1,|x'|)-\vartheta\e(|x'|)\rt]
\end{multline}
Let
\[
\s\e=\zeta\e\,\s^1\e+(1-\zeta\e)\,\s^2\e+\s^3\e.
\]
In force of equation~\eqref{divesigma3} and from construction of $\s^1\e$, $\s^2\e$ and $\zeta\e$ we have
\begin{align*}
  \nb\cdot\s\e&= \nb\cdot(\zeta\e\s^1\e)+\nb\cdot(1-\zeta\e)\s^2\e+\nb\cdot\s^3\e\\
  &=\zeta\e \nb\cdot\s^1\e+\zeta'\e( \ov\s\e^1-\vartheta\e)+\nb\cdot\s^3\e\\
  &=\zeta\e \nb\cdot\s^1\e=\nb\cdot(\s*\rho\e)
  \end{align*}
In addition for any $(x_1,x')$ such that $|x'|\geq r_*\eps$ from~\eqref{eq: corresponding normal flux} we derive 
\[
\s\e^3\cdot e_i(x_1,x')=-\zeta\e'(x_1)\;\frac{x_i}{|x'|^{n-2}}\;\int_{0}^{|x'|}s^{n-1}\lt[\ov\s\e^1(x_1,s)-\vartheta\e(s)\rt]\di s=0
\]
which justifies $\supp(\s\e)\subset I_{r_*\eps}$. Let us now prove 
\[
 \limsup_{\eps\dw0}\F_{\eps,a}(\s\e,u\e;\Om)\leq L\,f_a(m)+C\d.
\] 
\begin{figure}\centering
 \includegraphics[scale=.7]{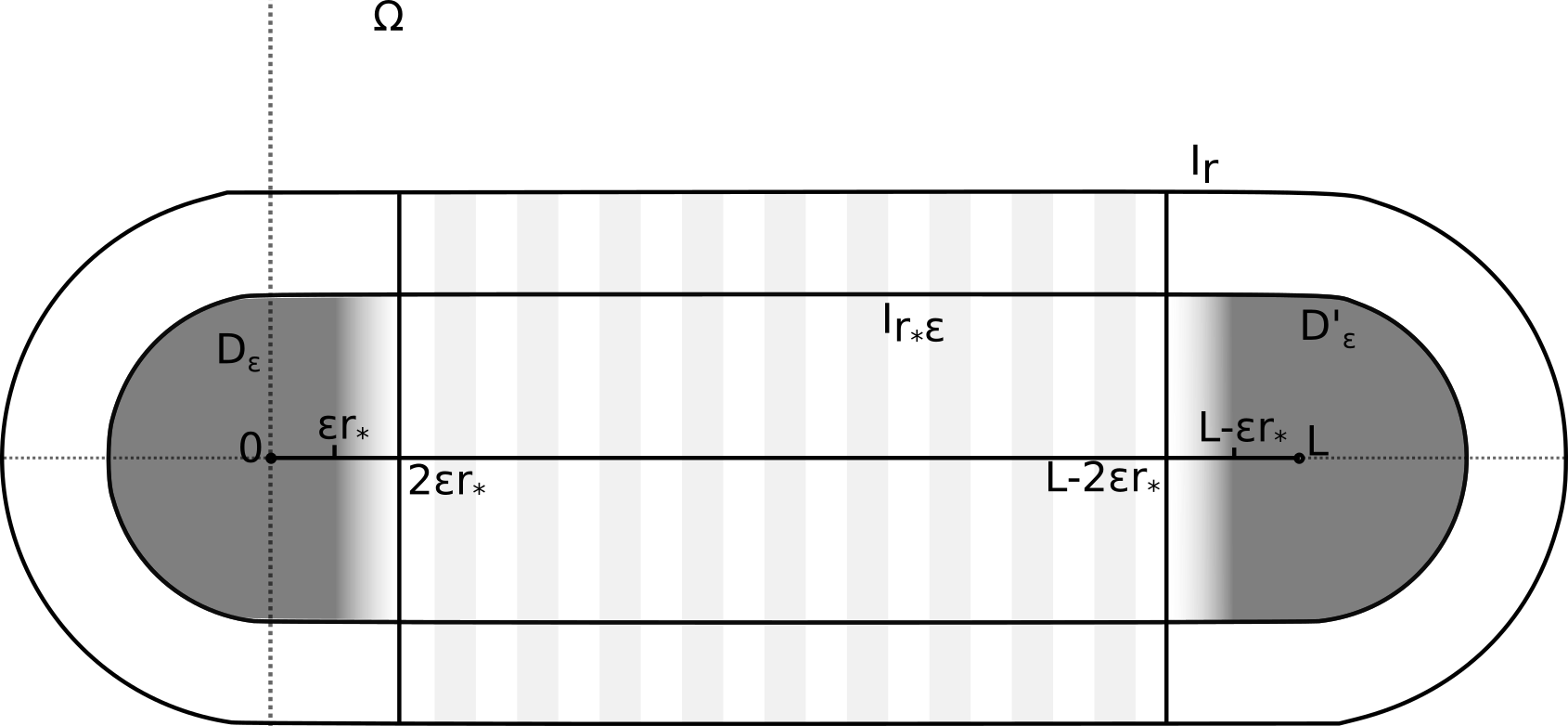}
 \caption{Illustration of the interval $I$ and both its $r$ and $(r_*\eps)$-enlargement for $r_*\geq1$.  In grayscale we plot the levels of the function $\zeta\e$, whilst the striped region corresponds to the cylinder $C_{r,\eps}$.}\label{figura4}
\end{figure}
We split $\Om$ as the union of  $\Om\sm I_r$, $C_{r,\eps}:=I_r\cap[2\;\eps,L-2\;\eps]\times \R^{n-1}$ and $D\e$ and $D'\e$, as show in figure~\ref{figura4}, where $D\e=\{x_1\leq 2\, r_*\,\eps\}\cap I_{ r_*\eps}$ and $D'\e=\{x_1\geq L- 2\, r_*\,\eps\}\cap I_{ r_*\eps}$. On $\Om\sm I_r$ we notice that $\s\e=0$ and $u\e=1$ therefore 
\[
\F_{\eps,a}(\s\e,u\e;\Om\sm I_r)=0.
\]
Observe that $|D\e|=|D'\e|=C\eps^n$, then we have the upper bound
\begin{equation*}\label{eq: estimationballs}
 \int_{D\e} |\s\e|^2\dx\leq 2\;\frac{m^2\; r_*^2}{\eps^{n-2}}\lt(\int_{B_1} \rho^2\dx+C\rt).
\end{equation*}
Taking into consideration this estimate we obtain 
\begin{equation}\label{ineq:Deps}
\F_{\eps,a}(\s\e,u\e;D\e)= \F_{\eps,a}(\s\e,u\e;D'\e)\leq \frac{(1-\eta)^2}{\eps^{n-1}}\mathscr{L}^n(D\e)+2\,m^2\, r_*^2 \;\frac{\eta}{\eps^{n-2}}.
\end{equation}
Finally on $C_{r,\eps}$ both $\s\e$ and $u\e$ are independent of $x_1$ and are radial in $x'$ then by Fubini's theorem and Proposition~\ref{prop: flimsup} we get
 \begin{equation*}
  \F_{\eps,a}(\s\e,u\e;C_{r,\eps})= \int_{2\,\eps r_*}^{L-2\,\eps r_*}\int_{B'_{r}}E_{\eps,a}(\vartheta\e,u\e)\leq \,L\,(f_a(m)+C\,\d) .
 \end{equation*}
Adding all together gives the desired estimate. It remains to discuss the case $r_*< 1$. From the point of view of the construction of $\s\e$ we need to replace the functions $\zeta\e$ with

\hspace{1.5cm}\begin{minipage}{.35\textwidth}
   \begin{gather*}
\qquad  \tilde\zeta\e(t)=1\qquad \mbox{ for } t\leq  \eps\mbox{ or } t\geq L- \eps,\\
\qquad  \tilde\zeta\e(t)=0\qquad \mbox{ for } 2\;  \eps\leq t\leq L-2\;  \eps,\\
\end{gather*} 
 \end{minipage}
  \begin{minipage}{.40\textwidth}
 \begin{equation*}
 \mbox{ and }\qquad \lt|\tilde\zeta\e'\rt|\leq \frac{1}{  \eps}.
 \end{equation*}
 \end{minipage}

 This choice ensures that $\s\e$ has all the properties previously obtained with $ r_*\,\eps$ replaced by $\eps$ accordingly. Define
 \[
 w\e(t):=
 \begin{dcases}
  \qquad\;\eta,&\mbox{ for }  t\leq \sqrt{3} \eps\\
  \frac{1-\eta}{r-\sqrt{3}}(t-\sqrt{3})+\eta,&\mbox{ for } \sqrt{3}\eps\leq  t\leq r.
\end{dcases}
 \]
 and set 
 \[
 u\e=\min\{\ov u\e(d(x,I)), w\e(|x|),w\e(|x-(L;0)|)\}.
 \]
 with these choices for $u\e$ and $\s\e$ the estimates follow analogously with small differences in the constants.

 \begin{figure}\centering
 \includegraphics[scale=.7]{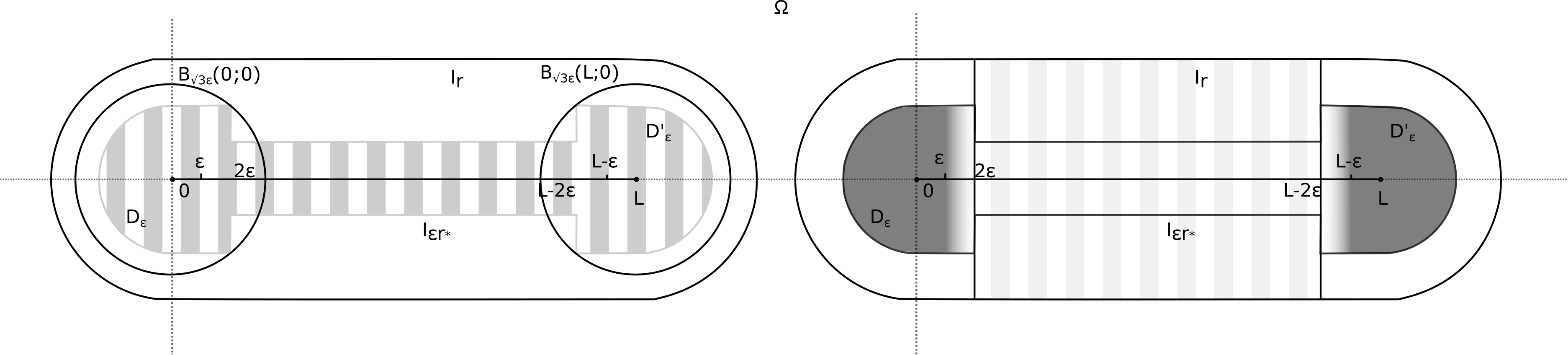}
 \caption{On the left the striped region corresponds to $\supp(\s\e)$, remark that the balls of radius $\sqrt{3}\eps$ centered respectively in $(0;0)$ and $(L;0)$ contain the modifications we have performed to satisfy the constraint. On the right we illustrate the level-lines of the cutoff function $\tilde \zeta\e$ in grayscale.}\label{figura5}
\end{figure}

%

 \medskip
 \textit{Step 2. (Case of a generic $\s$ in polyhedral form.) }Indeed, in force of the results quoted in Subsection~\ref{subSection: funcOnCurr} it is sufficient to show equation~\eqref{eq: 1limsup} for a polyhedral vector measure. Following the same notation introduced therein let 
 \begin{equation*}
  \s=\sum_{j=1}^N m_j \H^1\restr{\Sigma_j}\;\nu_j.
 \end{equation*}
 With no loss of generality we can assume that the segments $\Sigma_{j}$ intersect at most at their extremities. We consider measures $\s$ satisfying constraint~\eqref{point sources} so that if a point $P$ belongs to $\Sigma_{j_1},\dots,\Sigma_{j_P}$ it must satisfy of Kirchhoff law,
 \begin{equation}\label{Kirchhoff}
  \sum_{j_1}^{j_P} z_j\,m_{j}=
  \begin{dcases}
   c_i,& \mbox{ if } P\in\mathscr{S}.\\
   0,&\mbox{ otherwise. }
  \end{dcases}
 \end{equation}
 where $z_j$, is $+1$ if $P$ is the ending point of the segment $\Sigma_j$ with respect to its orientation, and $-1$ if it is the starting point. Let $\s\e^j$ and $u\e^j$ be the sequences constructed above for each segment $I_k$ and define 
 \begin{equation*}
  \s\e=\sum_{j=1}^N\s\e^j\qquad\mbox{ and }\qquad u\e=\min_j\lt\{u^j\e\rt\}.
 \end{equation*} 
 Let $P_j$ and $Q_j$ be respectively the initial and final point of the segment $\Sigma_j$ and recall that, by the construction made above, for each $j$
\[
\nb \cdot \s\e^j=m_j\lt(\d_{P_j}-\d_{Q_j}\rt)*\rho\e
\] 
then by linearity of the divergence operator, it holds
\[
\nb\cdot\s\e=\sum_{j=1}^N \nb \cdot\s\e^k=\sum_{j=1}^N m_j\, \lt(\d_{P_j}-\d_{Q_j}\rt)*\rho\e
\]
and the latter satisfies constraint~\eqref{point sources eps} in force of equation~\eqref{Kirchhoff}.
To conclude let us prove that
\begin{equation}\label{ineq:limsup}
\limsup_{\eps\dw 0} \F_{\eps,a} (\s\e,u\e;\Om)\leq \sum_{j=1}^N f_a(m_j)\H^1(\Sigma_j).
\end{equation}
Indeed the following inequality holds true
\begin{equation*}
\F_{\eps,a} (\s\e,u\e;\Om)\leq \sum_{j=1}^N\F_{\eps,a} (\s\e,u^j \e;\Om).
\end{equation*}
Suppose
\[
\supp(\s^{j_1}\e)\cap \supp(\s^{j_2}\e)\cap\dots\cap\supp{\s^{j_P}\e}\neq \void 
\] 
for some $j_1,\dots,j_P$ and all $\eps$. Let $r_*^{j_1},\dots,r_{*}^{j_P}$  be the radii  introduced above for each of these measures, let $\ov r_*= \max\{r_*^{j_1},\dots, r_*^{j_P},1\}$ , set $\ov m=\max\{m_{j_1},\dots,m_{j_P}\} $ and consider  $D_{j_1},\dots, D_{j_P}$ as defined previously. Since 
\[
\lt|\;\sum_{k=1}^{j_P}\s^k\e\;\rt|^2\leq C\sum_{k=1}^{j_P}\lt|\s^k\e\rt|^2 
\]
 and $u\e\leq u^j\e$ for any $j$, we have the following inequality
\[
\F_{\eps,a} (\s\e,u\e;\supp(\s^{j_1}\e)\cap\dots\cap \supp(\s^{j_P}\e))\leq C\sum_{k=j_1}^{j_P}  \F_{\eps,a} (\s^k\e,u^k\e;D_k)
\]
And by inequality~\eqref{ineq:Deps} follows
\[
\F_{\eps,a} (\s\e,u\e;\supp(\s^{j_1}\e)\cap\dots\cap \supp(\s^{j_P}\e))\leq C\lt(\frac{(1-\eta)^2}{\eps^{n-1}}\sum_{k=j_1}^{j_P}\mathscr{L}^n(D_k)+2\,\ov m^2\, \ov r_*^2 \;\frac{\eta}{\eps^{n-2}}\rt).
\]
Which vanishes as $\eps\dw0$. Let us remark that the intersection $\supp(\s^{j_1}\e)\cap \supp(\s^{j_2}\e)\cap\dots\cap\supp{\s^{j_P}\e}$ is non empty  for any $\eps $ only if the segments $\Sigma_{j_1}, \dots, \Sigma_{j_P} $ have a common point. Since we are considering a polyhedral vector measure composed by $N$ segments the worst case scenario is that we have $2N$ intersections in which at most $N$ segments intersects. We conclude
\begin{equation*}
\F_{\eps,a} (\s\e,u\e;\Om)\leq \sum_{j=1}^N\F_{\eps,a} (\s^j\e,u^j \e;\Om)+C(N)\lt(\frac{(1-\eta)^2}{\eps^{n-1}}\sum_{k=j_1}^{j_P}\mathscr{L}^n(D_k)+2\,\ov m^2\, \ov r_*^2 \;\frac{\eta}{\eps^{n-2}}\rt)
\end{equation*}
which, passing to the limit, yields inequality~\eqref{ineq:limsup}. 
\end{proof}

\section{The $k$-dimensional problem }\label{Section: kProblem}
\subsection{Setting}\label{subSection: ksetting}
Let $\s_0\in P_{k}(\Om)$ a polyhedral $k$-current with finite mass and let $\mathscr{S}:=\supp(\de\s_0)$ be compactly contained in $\Om$. We want to minimize a functional of the type~\eqref{Def: F} where the set of candidates ranges among all currents $\D_k(\ov\Om)$ such that 
\begin{equation*}\label{point sources k}
 \de \s =\de\s_0 \qquad\mbox{ in } \D^k(\R^n).
\end{equation*}
Let us introduce a parameter $\eta=\eta(\eps)$ which satisfies
\begin{equation}\label{Def: ak}
\eta(\eps)=a\eps^{n-k+1}\quad \mbox{ for }\quad a\in\R_+
\end{equation}
and let $X\e(\Om)$ be the set of couples $(\s\e,u\e)$ where $u\e\in W^{1,p}(\Om,[\eta,1])$ and has trace $1$ on $\de\Om$ and $\s\e$ is of finite mass with density absolutely continuous with respect to $\mathscr{L}^n$. In this case we identify the current $\s\e$ with its $L^1(\Om,\Lambda_k(\R^n))$ density. Furthermore as in equation~\eqref{point sources eps} given a convolution kernel $\rho\e$ we impose the constraint
\begin{equation*}\label{point sources eps k}
 \de\s\e=(\de \s_0)* \rho\e\qquad\mbox{ in } \D^k(\R^n).
\end{equation*}  
For $(\s\e,u\e)\in \D_k(\ov\Om)\times L^2(\Om)$ let 
\begin{equation}\label{Def: kFeps}
\F^k_{\eps,a}(\s\e, u\e;\Om)\, :=\begin{dcases}
                         \, \int_\Om\lt[ \eps^{p-n+k}  |\nb u\e|^p +\dfrac{(1-u\e)^2}{\eps^{n-k}} +
\dfrac{u\e|\sigma\e|^2}{\eps}\rt]\dx,&\mbox{if }  (\s\e,u\e)\in X\e(\Om),\\
\qquad\qquad\;\;+\oo,& \mbox{otherwise.}
                        \end{dcases}
\end{equation}
Let us denote with $X$ the set of couples $(\s,u)$ such that $\s$ is a $k$-rectifiable current satisfying~\eqref{point sources k} and $u\equiv1$. In this section we show that for any sequence $\eps\dw0$ the $\Gamma$-limit of the family $(\F_{\eps,a}^k)_{\eps\in\R_+}$ is the functional
\begin{equation}\label{Def: klimit}
 \F^k_a(\s,u;\ov\Om)=\begin{dcases}
             \int_{\supp{\sigma}} f^{n-k}_a(m(x))\dH^k(x),&\mbox{if }  (\s,u)\in X\\
             \qquad\qquad\;\;+\oo,&\mbox{otherwise in }\M(\Om,\R^n) \times L^2(\Om)
            \end{dcases}
\end{equation} 
Where the function $f^{n-k}_a:\R_+\rw\R_+$ is the function obtained in Appendix~\ref{appendix: Reduced Problem} for the choice $d=n-k$ and is endowed with the same properties stated for $f$ in Section~\ref{Section: Introduction}. In particular under the assumption $p>n-k$ we first prove a compactness theorem.
\begin{theorem}\label{teo: kcompactness}
Assume that $a>0$. For any sequence $\eps\dw0$, $(\s\e,u\e)\in \D_k(\ov\Om)\times L^2(\Om)$ such that 
\[
\F^k_{\eps,a}(\s\e,u\e;\Om)\leq F_0<+\oo
\]
then $u\e\rw1$ and there exists a rectifiable $k$-current $\s\in \D_k(\ov\Om)$ such that, up to a subsequence, $\s\e\rwstar\s$ and $(\s,1)\in X$. 
\end{theorem}
Then we show the $\Gamma$-convergence result, namely 
\begin{theorem}\label{teo:kGammalimit}
Assume that $a\geq0$.
\begin{enumerate}
 \item For any $(\s,u)\in\D_k(\ov\Om)\times L^2(\Om)$ and any sequence $(\s\e,u\e)\in\D_k(\ov\Om)\times L^2(\Om)$ such that $(\s\e,u\e)\rw(\s,u)$ it holds
\begin{equation*}
\liminf_{\eps\dw0}\F^k_{\eps,a}(\s\e,u\e;\Om)\geq \F^k_a(\s,u;\ov\Om).
\end{equation*}
\item For any couple $(\s,u)\in \D_k(\ov\Om)\times L^2(\Om)$ there exists a sequence $(\s\e,u\e)\in\D_k(\ov\Om)\times L^2(\Om)$ such that $(\s\e,u\e)\rw(\s,u)$ and 
\begin{equation*}
\limsup_{\eps\dw0}\F^k_{\eps,a}(\s\e,u\e;\Om)\leq \F^k_a(\s,u;\ov\Om).
\end{equation*}
\end{enumerate} 
\end{theorem}

\subsection{Compactness and $k$-rectifiability}\label{subSection: kcompactness}
\begin{proof}[Proof of Proposition~\ref{teo: kcompactness}.]
By the same procedure of Lemma~\ref{lemma: mass bound} we obtain 
\begin{equation}\label{eq: finitekmass}
|\s\e|(\Om)\leq\frac{F_0}{2}\;+\;\frac{F_0}{2\,a \,(1-\lambda)^2}\;+\;\sqrt{\frac{|\Om|\,\eps\, F_0}{\lambda}}
\end{equation}
and
\begin{equation*}
\int_{\Om}(1-u\e)^2\leq \eps^{n-k}\, F_0.
\end{equation*}
Therefore by the weak compactness of $\D_{k}(\Om)$ we obtain the existence of a limit $k$-current $\s$ a limit measure $\mu$ and a subsequence $\eps$ such that $\s\e\rwstar \s$, $|\s\e|\rwstar\mu$.  As in the $1$-dimensional case it is still necessary to prove the rectifiability of the limit current. This is obtained by showing that the support of $\s$ is of finite size.

\medskip 
\textit{Step 1. (Preliminaries and good representative for $v\in\Lambda_k(\R^n)$.) }\quad
Let us introduce the set
\begin{equation*}
\I:=\{I=(i_1,\dots,i_k): 1\leq i_1< i_2<\dots< i_k\leq n\},\qquad e_I=e_{i_1}\wedge \dots\wedge e_{i_k}
\end{equation*}
So that $\Lambda_k(\R^n)$ is the Euclidean space with basis $\{e_I\}_{I\in\I}$.
Let $v\in\Lambda_k(\R^n)$ and consider the problem
\begin{equation*}
a_0=\max\{a\in\R\,:\,v=a f_1\wedge\dots\wedge f_k+t\,:\,(f_1,\dots,f_n) \mbox{ orthonormal basis, } t\in (f_1\wedge\dots\wedge f_k)^\perp\}.
\end{equation*}
Notice that $a_0\geq1/\sqrt{|\I|}$. Assume that the optimum for the preceding problem is obtained with $(f_1,\dots,f_n)=(e_1,\dots,e_n)$. We note 
\begin{equation*}
v=a_0 e_{I_0}+\sum_{i\in\I_1}a_I e_I+\sum_{I\in\J} a_I e_I
\end{equation*}
with
\begin{equation*}
    I_0=e_1\wedge\dots\wedge e_k,\qquad \I_1:=\{I=(i_1,\dots,i_k)\in \I \,:\, 1\leq i_1<\dots< i_{k-1}\leq k<i_k\leq n\},\qquad \J:=\I\sm (\I_1\cap I_0).
\end{equation*}
We claim that $a_I=0$ for $I\in \I_1$. Indeed, let $I_1=(e_1,\dots,e_{l-1},e_{l+1},\dots,e_k,e_h)\in \I_1$ and for $\phi \in \R$, let $e^{\phi}$ be orthonormal base defined as 
\begin{equation*}
e_i = e^\phi_i \qquad \mbox{for } i\neq \{l,h\}, \qquad e_l=\cos(\phi)e^\phi_l-\sin(\phi) e^\phi_h, \qquad e_h=\sin(\phi)e^\phi_l+\cos(\phi) e^\phi_h.
\end{equation*} 
In this basis 
\begin{equation*}
v=\lt(a_0\cos(\phi)+a_{I_1}(-1)^{k-l}\sin(\phi)\rt) e^\phi_{I_0}+t^{\phi}, \quad \mbox{with} \quad w^\phi\in(e^\phi)^\perp.
\end{equation*}
By optimality of $(e_1,\dots,e_n)$ we deduce $a_{I_1}=0$ which proves the claim. Hence we write 
\begin{equation}\label{eq: localk1}
v=a_0 e_{I_0} +t,\quad \mbox{with } t\in \mbox{span}\{e_I\,:\,I\in\J\}.
\end{equation}
Now we let 
Let $\theta\in (0, 1/4^n)$ and $\Sigma$ be the set of points for which there exists a sequence $r_j\dw 0$ such that 
\begin{equation*}
\frac{\s(B_{r_j}(x))}{|\s|(B_{r_j}(x))}\longrightarrow w(x)\in S\Lambda_k(\R^n) \quad \mbox{and} \quad \frac{|\s|(B_{r_j/4}(x))}{|\s|(B_{r_j}(x))}\geq\theta.
\end{equation*}
In particular $w$ is a $|\s|$-measurable map and we have $\s=w\,|\s|\restr\Sigma$.

%
%

\medskip
\textit{Step 2. (Flux of $\s\e$ trough a small $(n-k)$-disk.)}\quad Consider a point $x\in \Sigma \sm \mathscr{S}$, with no loss of generality we assume $x=0$. Let $v=w(0)$, up to a change of basis, by equation~\eqref{eq: localk1} we write
\[
v=a_0 e_{I_0} +t,\quad \mbox{with } t\in \mbox{span}\{e_I\,:\,I\in\J\}.
\]
Let $j$ sufficiently small, such that $B_{r_j}\cap \mathscr{S}=\void $ and 
\begin{equation}\label{eq: localk2}
\s(B_{r_j})\cdot v\geq (1-\xi)|\s|(B_{r_j}).
\end{equation}
Set, to simplify notation, $r_j=r$ and $r_*=r/\sqrt{2}$. For $x\in \R^n$ we write $(x',x'')\in \R^{k}\times \R^{n-k}$ for the usual decomposition and denote $B'_r$, $B''_r$  the $k$-dimensional  and the $(n-k)$-dimensional ball respectively. Let $\chi\in C^{\oo}(B''_1)$ be a radial cut-off function with $\chi(x'')=1$ for $|x''|\leq 1/2$ and $\chi(x'')=0$ for $|x''|\geq 3/4$. Set $\chi_{r_*}(x'')=\chi(x''/r_*)$, then since $\s\e$ is a $L^1$ function for $\eps>0$ we can define 
\begin{equation}\label{eq: localk3}
g\e(x'):=\int_{B''_{r_*}}\chi_{r_*}(x'')\langle \s\e,e_{I_0}\rangle \dx''=\int_{B''_{r_*}} \chi_{r_*}(x'')\s\e^0\dx''
\end{equation}
for any $x'\in B'_{r_*}$.
Let us compute $\de_l g\e(x')$ for $l\in\{1,\dots,k\}$. 
Since $\de \s\e =0$ in $B_{r}$,  it holds $\langle \s\e,\di\om \rangle=0$ for any smooth $(k-1)$-differential form $\om\in \D^{k-1}(B_{r})$. Choosing $\om $ of the form 
\begin{equation}\label{differentialform}
 \om= \b(x) \dx_1\wedge \dots \dx_{l-1}\wedge \dx_{l+1}\wedge \dots \wedge \dx_k
 \end{equation} 
 we obtain 
 \begin{equation*}
 \di \om= (-1)^{l-1} \de_l \b(x) \dx_1\wedge \dots  \wedge \dx_k + (-1)^{k-1}\sum_{h=k+1}^d \de_h \b(x)\dx_1\wedge \dots \dx_{l-1}\wedge \dx_{l+1}\wedge \dots \wedge \dx_k\wedge \dx_{h}.
 \end{equation*}
 Denote $\sigma^I\e=\langle \sigma, e^I\rangle $, then imposing $\langle\s\e,\di\om\rangle=0$ for every $\b\in C^{\oo}_c(B_{r})$ in~\eqref{differentialform} yields 
\begin{equation*}
 (-1)^{k-l} \de_l \s^0\e+\sum\limits_{\substack{h\in\{k+1,\dots,d\}\\ I=(1,\dots,l-1,l+1,\dots,k,h)}}\de_h \s\e^I=0.
 \end{equation*} 
 Hence, 
 \begin{equation}\label{eq: localk4}
 \de_l g\e(x')=\frac{(-1)^{k-l}}{r_*}\sum\limits_{\substack{h\in\{k+1,\dots,d\}\\ I=(1,\dots,l-1,l+1,\dots,k,h)}}\int_{B''_{r_*}}\de_h\chi_{r_*}(x'')\s^I\e\dx''.
 \end{equation}
Let us introduce the notation
\begin{equation*}
 \s\e^{\I_1}:=\sum_{I\in\I_1} \s\e^I\,e_I,
 \end{equation*} 
denoting with $\nb'$ the gradient with respect to $x'$, equation~\eqref{eq: localk4} rewrites as 
 \begin{equation}\label{eq: localk5}
 \nb'g\e(x')=\frac{1}{r_*}\int_{B''_{r_*}} Y\lt(\frac{x}{r_*}\rt) \s\e^{\I_1}\dx''.
 \end{equation}
Where $Y$ is smooth and compactly supported in $B''_1$ and with values into the linear maps : $\mbox{span}\{e_I:I\in\I_1\}\rw\R^k$. Let us prove that, for some $\hat r$, the functions $g\e$ converge in BV-$*$ to some $g$. First for a.e. choice of $\hat r\in[(1-\xi)r_*,r_*]$ it must hold $\mu(\de B'_{r_*}\times B''_{r_*})=0$ so that
\begin{equation}\label{eq: localk6}
g\e(x')=\int_{B''_{r_*}}\chi_{r_*}(x'')\langle \s\e,e_{I_0}\rangle \dx''\xrightarrow[]{\eps\dw0} \int_{B''_{r_*}}\chi_{r_*}(x'')\di\langle \s,e_{I_0}\rangle =: g(x').
\end{equation}
Secondly we define the mean value  
\begin{equation*}
  \ov g:= \frac{1}{|B'_{\hat r}|}\int_{B'_{\hat r}} g(x')\dx'=\frac{1}{|B'_{\hat r}|}\int_{B'_{\hat r}}\lt[\int_{ B''_{r_*}} \chi_{r_*}(x'')\,\di\s^0\rt]\dx'.
  \end{equation*}
and taking advantage of~\eqref{eq: localk2} and the definition of $\Sigma$, we see that 
  \[
  \ov g\geq\lt( \frac{\theta }{\sqrt{|\I|}}-\xi\rt)\frac{|\s|(B_{r})}{|B'_{\hat r}|}>0.
  \]
   On the other hand, denoting $\Pi:\R^n\to \R^{n-k}$, $x\mapsto x''$, from~\eqref{eq: localk1}, we have 
   \begin{equation*}
    |\Pi\s|(B'_{\hat r}\times B''_{r_*})\leq \sqrt{3\xi}\lt(\frac{\theta}{\sqrt{|\I|}}-\xi\rt)\;|B'_{\hat r}|\; \ov g.
    \end{equation*}
Now from \eqref{eq: localk5}~-~-~\eqref{eq: localk6} and the latter we obtain 
\begin{equation*}
\langle D'g,\phi\rangle =\frac{1}{r_*}\int_{B'_{\hat r}\times B''_{r_*}}\phi(x') \;Y\lt(\frac{x''}{r_*}\rt)\,\di\s^{\I_1}\quad\mbox{ and }\quad |D'g|(B'_{\hat r})\leq \frac{C\,|B'_{\hat r}|\,\sqrt{\xi}\,\ov g}{r_*}.
\end{equation*}
Finally from Poincar\'e - Wirtinger inequality and the convergence $g\e\rw g$ in $L^1(B'_{\hat r})$ is easy to show that for any sufficiently small $\eps$ the sets
\[
A\e=\lt\{x\in B_{\hat r} \;: \, g\e(x)\geq  \frac{\ov g}{8}\rt\}
\]
are such that $|A\e|\geq |B'_{\hat r}|/2$.

\medskip
\textit{Step 3. (Conclusion.)} \quad Set $\vartheta\e(x',x'')=|\chi_{r_*}(x'')\s^0\e|$ and observe that for fixed $x'$ by construction 
\[
\int_{B_{r_*}} \vartheta\e(x',x'')\dx''=g\e(x').
\]
Therefore for any $x'\in A\e$ it holds $\int_{B_{r_*}} \vartheta\e(x',x'')\dx''\geq \ov g/8$. Furthermore $\supp(\vartheta\e(x'))\subset B'_{\tilde r}$ with $\tilde r :=\frac{3}{4}{r_*}<{r_*}$. 
Now, by Fubini 
\begin{equation}
\begin{aligned}
\F^k_{\eps,a}(\s\e, u\e;B_{r})\, &\geq \int_{A\e}\int_{B''_{r}}\lt[ \eps^{p-n+k}  |\nb u\e|^p +\dfrac{(1-u\e)^2}{\eps^{n-k}} +\dfrac{u\e|\sigma\e|^2}{\eps}\rt]\dx''\dx'\\
&\geq \int_{A\e}\int_{B''_{r_*}}\lt[ \eps^{p-n+k}  |\nb u\e|^p +\dfrac{(1-u\e)^2}{\eps^{n-k}} +\dfrac{u\e|\vartheta\e(x',x'')|^2}{\eps}\rt]\dx''\dx'
\end{aligned}
\end{equation}
With the notation introduced in Subsection~\ref{subSection: reducedProblem} and by defintion of $A\e$ 
\begin{equation*}
\F_{\eps,a}(\s\e,u\e;B_{r})\, \geq\,\int_{A\e}\inf_{(\vartheta,u)\in \ov Y_{\eps,a}(m,r)^{\tilde r}(\ov g/8,r)} E^k_{\eps,a}(\vartheta,u)\, \dx'=\,\int_{A\e}f^{\tilde r}\e\lt(\ov g/8\rt)\, \dx'=f^{\tilde r}\e\lt(\ov g/8\rt)|A\e|.
\end{equation*}
Taking the infimum limit, by Proposition~\ref{prop: structurefm}, in particular equation~\eqref{eq: reducedconstant} we get
\begin{equation}\label{ineq: kfinitelenght}
\liminf_{\eps\dw0} \F^k_{\eps,a}(\s\e,u\e;B_{r})\geq \liminf_{\eps\dw 0}f^{\tilde r}\e\lt(\ov g/8\rt)|A\e|\,\geq\,\kappa\; \frac{|B'_{\hat r}|}{2}.
\end{equation}
Recall that the latter stands for a.e.  $\hat r\in [(1-\xi)r_*,r_*]$ and $r_*=r/\sqrt{2}$ thus we may rewrite
\begin{equation*}
\liminf_{\eps\dw0} \F^k_{\eps,a}(\s\e,u\e;B_{r})\,\geq \,\kappa\; \frac{\om_{k}\;r^k}{2^{1+k/2}}.
\end{equation*}
As in Lemma~\ref{lemma: local length Sigma} we conclude applying Besicovitch theorem  to obtain $\H^k(\Sigma)<+\oo$. Finally, thanks to the latter and equation~\eqref{eq: finitekmass},  Theorem~\ref{teo: krectifiable} applies and $\s$ is a $k$-rectifiable current. 
\end{proof}

 \subsection{$\Gamma$-liminf inequality}\label{subSection: kliminf}
\begin{proof}[Proof of item 1) of Theorem~\ref{teo:kGammalimit}.]
 With no loss of generality we assume that $\liminf_{\eps\dw0} \F^k_{\eps,a}(\s\e,u\e)<+\oo$ otherwise the inequality is trivial. For a Borel set $A\subset \Omega$, we define
 \[
 H^k(A)\,:=\,\liminf_{\eps\dw0}\F^k_{\eps,a}(\s\e,u\e;A),
 \] 
so that $H^k$ is a subadditive set function. By assumption, the limit current $\sigma$ is $k$-rectifiable; we write $\s= m\,\nu\, \H^k\restr \Sigma$. We claim that 
\begin{equation}\label{eq: liminflocalek}
\liminf_{r\dw0}\dfrac{H^k\lt(\ov{B(x,r)}\rt)}{\om_k\, r^k} \, \geq\,f^{n-k}_a(m(x))\qquad\mbox{for $\H^k$-almost every $x\in\Sigma$.}
 \end{equation}
 Assuming the latter the proof is achieved as in Theorem~\ref{teo: 1Gammaliminf}.
To establish the claim~\eqref{eq: liminflocalek} we restrict our attention to a single point and we assume $x=0$, $m=m(0)$ and  $\nu(0)=e_1\wedge\dots\wedge e_k$ then for any $\xi>0$ there exists $r_0=r(\xi)$ such that
 \begin{equation}\label{eq: kstrongineq}
 \langle\s, e_1\wedge\dots\wedge e_k\rangle(B_r) \geq (1-\xi) |\s|(B_r)\quad \mbox{and} \quad (1-\xi)\, m\,\leq \,\frac{|\s|(B_r)}{\om_k r^k}\,\leq \,(1+\xi)\, m, \qquad \mbox{for } r\leq r_0.
 \end{equation} 
Let $\d$ be an infinitesimal quantity and set, for $r<r_0$, $\hat r=\sqrt{1-\d^2}\; r$ and $\tilde  r = \d r$ and define the cylinder 
 \begin{equation*}
C_{\d,r}(e_1,\wedge\dots\wedge e_n)=C_{\d,r}:=\lt\{(x';x'')\in\R^k\times \R^{n-k}:|x'|\leq\hat r \;\mbox{and}\; |x''|\leq {\tilde r}\rt\}.
  \end{equation*}
  Let $\chi(x'')$ be the radial cutoff introduced in the previous proposition and set $\chi_{\tilde r }(x'')=\chi(x''/\tilde r) $, $\s\e^0=\langle\s\e, e_1\wedge\dots\wedge e_k\rangle$ and for any $x'\in B'_{\hat r}$ set 
    \begin{equation*}
g\e(x'):=\int_{B''_{\tilde r}}\chi_{\tilde r}(x'')\di\langle \s\e,e_{I_0}\rangle =\int_{B''_{\tilde r}} \chi_{\tilde r}(x'')\di\s\e^0,
  \end{equation*}
  as in equation~\eqref{eq: localk3}. Up to a smaller choice for $r_0$ we can assume $B_{r}\cap \mathscr{S}=\void$ therefore $\de \s\restr B_{r}=0$, and from equations~\eqref{eq: localk3}~-~\eqref{eq: localk5} it holds
  \begin{equation*}
 \nb'g\e(x')=\frac{1}{\tilde r}\int_{B''_{\tilde r}} Y\lt(\frac{x}{\tilde r}\rt) \di\s\e^{\I_1}.
 \end{equation*}
 For a.e. choice of $\d$ it holds $|\s|(\de B'_{\hat r}\times B''_{\tilde r})=0 $ therefore, for any such choice, $\g\e$ converges in $BV(B_{\hat r})$ to
\begin{equation*}
g(x'):=\int_{B''_{\tilde r}}\chi_{\tilde r}(x'') \di\s^0 \qquad\mbox{and} \qquad\langle D'g,\phi\rangle =\frac{1}{\tilde r}\int_{B'_{\hat r}\times B''_{\tilde r}}\phi(x') \;Y\lt(\frac{x''}{\tilde r}\rt)\,\di\s^{\I_1}.
\end{equation*}
Now we use~\eqref{eq: kstrongineq} to improve the estimates on $\ov g$ and $|D'g|$. Indeed, for $\d$ sufficiently small,  $\tilde r<\hat r /2$ therefore $B_{\tilde r}\subset B'_{\hat r}\times B''_{\tilde r}$ and
\begin{align*}
 \lim_{\eps\dw0} \ov g\e \geq (1-\xi) \frac{1}{|B''_{\hat r}|}\int_{B'_{\hat r}\times B''_{\tilde r}} \chi_{{r_*}}(x')\di|\s|\geq (1-\xi)^2 m.
\end{align*}
and denoting $\Pi:\R^n\to\R^{n-k}$, $x\mapsto x''$ we have
\[
|\Pi\s|(C_{r})\leq (1+\xi) \sqrt{3\xi}\;|B'_{\hat r}|\; m \quad\mbox{ and }\quad |D'g|(B'_{\hat r})\leq \frac{C\,|B'_{\hat r}|\,\sqrt{\xi}\,m}{\tilde r}.
\]
Choose $r$ sufficiently small then by Poincar\'e - Wirtinger inequality there exists a set $A$ of almost full measure in $B_{\hat r}$ such that $g\e(x')\geq (1-\xi)^2\,m$, and following the proof of the previous lemma (Step 3) up to equation~\eqref{ineq: kfinitelenght} we get
\begin{equation*}
\liminf_{\eps\dw0} \F^k_{\eps,a}(\s\e,u\e;B_{r})\geq \liminf_{\eps\dw 0}f^{n-k}_{\eps,a}\lt( (1-\xi)^2\,m,r,\tilde r\rt)|A|.
\end{equation*}
Since $\xi$ and $\d$ are arbitrary and $|A|$ can be chosen arbitrary close to $|B_{\hat r}|$ applying Proposition~\ref{prop: structurefm}  with $d=n-k$ to  the latter we conclude
\begin{equation*}
 \liminf_{\eps\dw0} \F^k_{\eps,a}(\s\e,u\e;B_{r})\geq f^{n-k}_a\lt(m\rt)\om_{k}r^k.
\end{equation*}
\end{proof} 

\subsection{$\Gamma$-limsup inequality}\label{subSection: klimsup}
For the lim-sup inequality, we start by approximating $\s$ with a polyhedral current: given $\delta>0$, there exists a $k$ polyhedral current $\tilde\s$ satisfying $\pt\tilde\s=\de\s_0$ and with $\mathbb{F}(\tilde \sigma-\sigma)<\delta$ and $\F_a(\tilde \sigma)<\F_a(\sigma)+\eps$. This result of independent interest is established in~\cite{RepresentationCFM}. A similar result has been proved recently by Colombo et al. in~\cite[Prop. 2.6]{CoDeMaSt17} (see also~\cite[Section~6]{White1}). The authors build an approximation of a $k$-rectifiable current in flat norm and in energy but their construction creates new boundaries and can not ensure the condition $\pt \sigma=\pt \sigma_0$. 

\begin{proof}[Proof of item 2) of Theorem~\ref{teo:kGammalimit}:] ~\\
By~\cite[Theorem 1.1 and Remark 1.6]{RepresentationCFM} we can assume that $\s$ is a polyhedral current. We show how to produce the approximating $(\s\e,u\e)$ for $\s$ supported on a single $k-$dimensional simplex $Q$. We assume with no loss of generality that $Q\subset \R^k$, and that $\s$ writes as
\begin{equation*}
  m\;\H^{k}\restr_{Q} \wedge (e_1\wedge \dots \wedge e_k).  
\end{equation*}
For $\d>0$ fixed, we consider the optimal profiles
\begin{equation*}
\ov u\e(t):=
\begin{dcases}
\eta,&\mbox{ for } 0\leq t\leq  r_*\eps,\\
v_\d\lt(\frac{t}{\eps}\rt),&\mbox{ for } r_*\eps\leq  t\leq r,\\
1&\mbox{ for } r\leq  t,
\end{dcases}
\qquad\mbox{and}\qquad\vartheta\e=\frac{m\;\chi_{B''_{r_*\eps}}(x'')}{\;\om_{n-k}\;   (\eps r_*)^{n-k}}
\end{equation*}
with $r_*$ and $ v_\d$, defined in Proposition~\ref{prop: flimsup} for the choice $d=n-k$. We denote $\de Q$ the relative boundary of $Q$ and given a set $S$ we write $d(x,S)$ for the distance function from $S$. Recall that we use the notation $S_t$ for the $t$-enlargement of the set $S$ and $S'$ to denote its projection into $\R^k$. We first assume, as did for the case $k=1$, $r_*\geq1$, and introduce $\zeta\e$ a $0$-form depending on the first $k$ variables $x'$, satisfying 
  \begin{gather*}
 \zeta\e(x')=1,\qquad \mbox{ for } x'\in (\de Q)'_{  r_* \eps}:=\lt\{x\in \Om\;:\;d(x',\de Q)\leq r_* \eps\rt\},\\
 \zeta\e(x')=0,\qquad \mbox{ for } x'\in \Om \sm (\de Q)'_{2  r_* \eps},\\
 \lt|\di\zeta\e\rt|\leq \frac{1}{ r_* \eps}.
\end{gather*} 
Then we proceed by steps, first set $\ov \s^1\e:=(|\s|*\rho\e)$
  \begin{equation*}
   \s\e^1=\ov \s^1\e e_1\wedge \dots\wedge e_k\qquad \mbox{and}\qquad \s\e^2(x',x'')=\vartheta\e(|x''|)\, \wedge(e_1\wedge\dots\wedge e_k).
  \end{equation*}
  and observe that $\supp(\s\e^1)\cup\supp(\s\e^2)\subset Q_{r_*\eps}$, both $\s\e^1$ and $\s\e^2$ are radial in $x''$ and with a small abuse of notation we denote $\ov \s^1\e(x',s)=\ov \s^1\e(x',|x''|)$, finally for any $x'$
  \[
  \int_{\{x'\}\times B''_{r_*\eps}}[\ov\s^1\e(x',|x''|)-\vartheta\e(|x''|)]\dx''=0.
  \]
  Now we take advantage of $\zeta\e$ in order to interpolate between $\s^1\e$ and $\s^2\e$, note that such interpolation may affect the boundary of the new current therefore we first introduce $\s\e^3$ which corrects this defect. In particular set
  \begin{equation*}
 \s\e^3(x',x'')=-\sum_{i=k+1}^n\lt[\frac{x_i}{|x''|^{n-k}}\;\int_{0}^{|x''|}s^{n-k-1}\lt[\ov \s^1\e(x',s) \vartheta\e(s)\rt]\restr \di \zeta\e\di s\rt]\wedge e_i,
\end{equation*}
  and
  \begin{equation*}
 \s\e=\s\e^1\restr \zeta\e+\s\e^2\restr(1-\zeta\e)+\s\e^3.
\end{equation*}
  With this choice by a calculation similar to equation~\eqref{divesigma3} it holds 
  \begin{gather*}
  \de\s\e= -\de \s*\rho\e\restr \zeta\e-\s^1\e\restr \di \zeta\e-\underbrace{\de \s\e^2\restr(1-\zeta\e)}_{=0}+\s^2\e\restr \di \zeta\e+\de  \s^3\e=(\de \s)*\rho\e.
\end{gather*}
  On the other hand the phase-field is simply defined as $u\e(x)=\ov u\e(d(x,Q))$. In the case $r_*<1$ we need to modify the construction. For $\s\e$ it is sufficient to replace every occurrence of $\zeta\e$ with $\tilde \zeta\e$, which satisfies
   \begin{gather*}
 \tilde \zeta\e(x')=1,\qquad \mbox{ for } x'\in (\de Q)'_{\eps}:=\lt\{x\in \Om\;:\;d(x',\de Q)\leq \eps\rt\},\\
 \tilde\zeta\e(x')=0,\qquad \mbox{ for } x'\in \Om \sm (\de Q)'_{2  \eps},\\
 \lt|\di\tilde \zeta\e\rt|\leq \frac{1}{\eps}.
\end{gather*} 
Now let 
 \[
 w\e(t):=
 \begin{dcases}
  \qquad\;\eta,&\mbox{ for }  t\leq \sqrt{3} \eps,\\
  \frac{1-\eta}{r-\sqrt{3}}(t-\sqrt{3})+\eta,&\mbox{ for } \sqrt{3}\eps\leq  t\leq r.
\end{dcases}
 \]
 and set 
 \[
 u\e=\min\{\ov u\e(d(x,Q)), w\e(d(x,\de Q))\}.
 \]
 \begin{remark} Given a polyhedral current $\s$ such that $\de\s=\de\s_0$ we perform our construction on each simplex and define $\s\e$ as the sum of these elements. The linearity of the boundary operator grants that $\de\s\e=\de\s_0*\rho\e$. The phase field is chosen as the pointwise minimum of the local phase fields. Finally the estimation for the $\Gamma$-limsup inequality is achieved in the same manner as Theorem~\ref{teo: 1Gammalimsup}.
\end{remark}
\end{proof}

\section{Discussion about the results}\label{sec: discussion}
By Lemma~\ref{lemmaA:pointwiseconvergence} for any fixed $d = n-k$ the cost function $f^d_a$ pointwise converges as $a\dw0$ to the function
\[
f(m)=\begin{dcases}
          \kappa,&\mbox{ for }m>0,\\
          0,\mbox{ if }m=0,
         \end{dcases}
\]
where $\kappa$ is the constant value obtained in Proposition~\ref{prop: structurefm} and depends on $d$.
This condition is sufficient to prove that the family of functionals $\F^k_a$, parametrized in $a$,  $\Gamma$-converges to the functional 
\[
\F^k(\s;\Om):=\begin{dcases}
             \kappa\;\H^k(\Sigma\cap \Om),&\mbox{ for }\s=m\,\nu\,\H^k\restr\Sigma,\\             +\oo,&\mbox{otherwise}.
            \end{dcases}
\]
As a matter of fact for any sequence $\s_a\rwstar\s$ in $\D_k(\Om)$ it holds
\[
\liminf_{a\dw0} \F^k_a(\s;\Om)\geq \F^k(\s;\Om)
\]
since $f^d_a(m)\geq \kappa$. On the other hand setting $\s_a:=\s$ we construct a recovery sequence for any $\s$ and obtain the $\Gamma$-limsup inequality
\[
\limsup_{a\dw0}\F^k_a(\s_a;\Om)=\limsup_{a\dw0}\F^k_a(\s;\Om)=\F^k(\s;\Om).
\]
This allows to interpret our result as an approximation of the Plateau problem in any dimension and co-dimension.

\appendix
\section{Reduced problem in dimension $n-k$}\label{appendix: Reduced Problem}
\subsection{Auxiliary problem}\label{subSection: auxiliaryproblem}
In this appendix we show the results previously enunciated in Subsection~\ref{subSection: funcOnCurr}, with the notation introduced therein let us define the auxiliary set 
\begin{equation*}
\ov Y_{\eps,a}(m,r)=\lt\{(\vartheta,u)\in L^2(B_r)\times W^{1,p}(B_r,[\eta,1])\,:\, \|\vartheta\|_1= m\mbox{ and } u_{|\de B_r}\equiv1\rt\},
\end{equation*} 
and the associated minimization problem
 \begin{equation}\label{def:min2}
  \ov f^d_{\eps,a}(m,r)=\inf_{ \ov Y_{\eps,a}(m,r)}E_{\eps,a}(\vartheta,u;B_r).
 \end{equation}
 First we show that both $f^d_{\eps,a}(m,r,\tilde r)$ and $\ov f^d_{\eps,a}(m,r)$ are bounded by the same constant as $\eps\dw0$ and that the value of the second term is achieved by a radially symmetric couple of $\ov Y_{\eps,a}(m,r)$. These two facts are then used to show that for each $m$ the limit values of $\ov f^d_{\eps,a}(m,r)$ and $f^d_{\eps,a}(m,r,\tilde r)$ as $\eps\dw 0$ are equal and independent of the choices $(r,\tilde r)$ to the extent that $0<\tilde r<r$. Let us start by showing the first two properties.
 \begin{lemma}\label{lemmaA: radial symmetry}
For each $\eps$, $m>0$ and $r>0$
\begin{enumerate}[a)]
 \item there exists a constant $C=C(m)\leq C_0\sqrt{1+m^2}$ such that 
\begin{equation}\label{eq: upboundd}
 f^d_{\eps,a}(m,r,\tilde r)<C\qquad\mbox{ and }\qquad \ov f^d_{\eps,a}(m,r)<C.
\end{equation}
\item Both the problem defined in equation~\eqref{def:min1} and equation~\eqref{def:min2} admit a minimizer. Moreover among the minimizers of $E_{\eps,a}$ in $\ov Y_{\eps,a}(m,r)$ it is possible to choose a radially symmetric couple $(\vartheta\e,u\e)$ such that $u\e$ is radially non-decreasing and $\vartheta\e$ is radially non-increasing.
\end{enumerate}
\end{lemma}
\begin{proof}
\begin{enumerate}[a)]
 \item To show the bound it is sufficient to define
\begin{equation*}
 u\e(x)\,:=\, \begin{cases}
\quad \eta& \mbox{ if } |x|< r_1\eps,\\
\eta+\dfrac{1-\eta}{(r_2- r_1)\eps}(|x|- r_1\eps) & \mbox{ if }  r_1\eps\leq |x|< r_2\eps,\\
\quad1 & \mbox{ if }   r_2\eps\leq |x|<r,
\end{cases}
\qquad \vartheta\e(x)\,:=\, 
 \begin{cases}
\dfrac m{|B_{r_1\eps}|}& \mbox{ if } |x|<r_1\eps,\\
\quad 0 & \mbox{ if } r_1\eps\leq |x|< r.
\end{cases}
\end{equation*}
Evaluating the energy we get, for any choice of $r_1<r_2<r$,
\begin{equation*}
 E_{\eps,a}(u\e,\vartheta\e)\leq \frac{a\, m^2}{\;\om_{d}\;\, r_1^{d}}+\;\om_{d}\;\lt[ r_1^{d}+\frac{1}{(r_2-r_1)^2}\lt(\frac{r_2^{d}-r_1^{d}}{d}\,r_2^2-\frac{r_2^{d+1}-r_1^{d+1}}{d+1}\,2\,r_2+\frac{r_2^{d+2}-r_1^{d+2}}{d+2}\rt)\rt].
\end{equation*}
As soon as $r_1\eps<\tilde r$, we have $(\vartheta\e,u\e)\in Y_{\eps,a}(m,r,\tilde r) \cap \ov Y_{\eps,a}(m,r)$. Choosing $r_1=(\sqrt{a}m)^{1/d}$ and $r_2=(1+\sqrt{a}m)^{1/d}$, we get
\begin{equation*}
\max\{f^d_{\eps,a}(m,r,\tilde r),\ov f^d_{\eps,a}(m,r)\}\leq C_0\sqrt{1+m^2}.
\end{equation*}
\item To show the existence of minimizers for both minimization problems we use the direct method of the Calculus of Variation. The lower semicontinuity of the integral with integrand $u|\vartheta|^2$ is ensured by Ioffe's theorem~\cite[theorem 5.8]{Am_Fu_Pal}.
Now given any minimizing couple $(\hat \vartheta\e,\hat u\e)\in \ov Y_{\eps,a}(m,r)$, let $\vartheta\e$ be the decreasing Steiner rearrangement of $\hat \vartheta\e$ and $u\e$ the increasing rearrangement of $\hat u\e$. Indeed, since $\hat u\e$ has range in $[\eta,1]$, we still have $u\e~_{|\de B_r}\equiv1$. Polya's Szego and Hardy-Littlewood's inequalities ensure 
\[E_{\eps,a}(\vartheta\e,u\e)\leq E_{\eps,a}(\hat \vartheta\e,\hat u\e)\]
\end{enumerate}
\end{proof}
Let us prove the asymptotic equivalence of the values $f^d_{\eps,a}(m,r,\tilde r)$ and $\ov f^d_{\eps,a}(m,r)$ as $\eps\dw0$. 
\begin{lemma}[Equivalence of the two problems]\label{lemmaA: equivalenceE}
 For any $\tilde r<r$ and $m>0$ it holds
 \begin{equation*}
  |f^d_{\eps,a}(m,r,\tilde r)- \ov f^d_{\eps,a}(m,r)|\,\stackrel{\eps\dw0}\longto\, 0
 \end{equation*}
\end{lemma}
\begin{proof}
\textit{Step 1: [$f^d_{\eps,a}(m,r,\tilde r)\leq \ov f^d_{\eps,a}(m,r)+O(1)$]}~\\
Consider for each $\eps$ the radially symmetric and monotone couple $(\vartheta\e,u\e)\in \ov Y_{\eps,a}(m,r)$ as introduced in the previous lemma. Take $\xi \in(\eta,1)$ and let us set 
\begin{equation}\label{eq: rx}
 r\x:=\sup\{t\in (0,r):u\e(t)\leq \xi\}\quad\mbox{ with $r\x=0$ if the set is empty}.
\end{equation}
By Cauchy-Schwartz inequality it holds 
\begin{equation*}
 C\geq\frac{\int_{B_{r}\sm B_{r\x}}u\e|\vartheta\e|^2\dx}{\eps}\geq\xi\; \frac{\lt(\int_{B_{r}\sm B_{r\x}}|\vartheta\e|\dx\rt)^2}{\;\om_{d}\; r^{d} \eps}.
\end{equation*}
Let us define $\Delta\x:= \int_{B_{r}\sm B_{r\x}}|\vartheta\e|$, the latter ensures that $\Delta\x\in o(\eps^{\a/2})$. Let us now set $\hat \vartheta\e =\lt(\frac{m\vartheta\e}{\int_{B_{r\x}}\vartheta\e}\rt)\ind_{B_{r\x}}$ which is not null for $\eps$ small. We have $(\hat\vartheta\e,u\e)\in Y_{\eps,a}(m,r,\tilde r)$ if and only if $r\x\leq\tilde r$. Indeed, this holds as 
\begin{equation}\label{eq: rxoeps}
 C\geq\int_{B'_{r\x}}\frac{(1-u\e)^2}{\eps^{d}}\dx\geq \;\om_{d}\;(1-\xi)^2\,\lt(\frac{r\x}{\eps}\rt)^{d},
\end{equation}
which ensures that $r\x= O(\eps)$. Finally let us evaluate the energy
\begin{align*}
E_{\eps,a}(\hat\vartheta\e,u\e)&=\int_{B_r}\lt[\eps^{p-d}|\nb u\e|^p+\frac{(1-u\e)^2}{\eps^{d}}+\frac{u\e|\hat\vartheta\e|^2}{\eps}\rt]\dx\\
&= \int_{B_r}\lt[\eps^{p-d}|\nb u\e|^p+\frac{(1-u\e)^2}{\eps^{d}}\rt]\dx+\int_{B_{r\x}}\frac{u\e\,m^2\,|\vartheta\e|^2}{\eps(\int_{B_{r\x}}\vartheta\e)^2}\dx\\
&\leq\;  \frac{m^2\;\om_{d}\;}{\lt(\int_{B_{r\x}}\vartheta\e\rt)^2}E_{\eps,a}(\vartheta\e,u\e)=[1+O(1)]E_{\eps,a}(\vartheta\e,u\e).
\end{align*}
Passing to the infimum we get
\begin{equation}\label{eq: equivalence1}
 f^d_{\eps,a}(m,r,\tilde r)\leq \ov f^d_{\eps,a}(m,r)+O(1).
\end{equation} 

\medskip
\textit{Step 2: [$\ov f^d_{\eps,a}(m,r)\leq f^d_{\eps,a}(m,r,\tilde r)+O(1)$]}~\\
Consider a minimizing couple $(\vartheta\e,u\e)$ such that
\begin{equation*}
 f^d_{\eps,a}(m,r,\tilde r)= E_{\eps,a}(\vartheta\e,u\e).
\end{equation*}
Let $\chi$ be a smooth cutoff function such that $\chi(x)=1$ if $|x|\leq \tilde r$ and $\chi(x)=0$ if $|x|>\frac{r+\tilde r}{2}$ and set $v\e=\chi u\e+(1-\chi)$. By construction $(\vartheta\e,v\e)\in \ov Y_{\eps,a}(m,r)$, furthermore, since $u\e\in (0,1]$, it holds that $u\e\leq v\e$ and $(1-u\e)^2\geq(1-v\e)^2$.  Moreover as $v\e\equiv u\e$ on $B_{\tilde r}$ we have $\int_{B_r} u\e|\vartheta\e|^2\dx=\int_{B_r} v\e|\vartheta\e|^2\dx$. Eventually, we estimate the gradient component of the energy as follows
 \begin{align*}
\int_{B_r}\eps^{p-d}|\nb v\e|^p\dx&=\int_{B_r}\eps^{p-d}|\chi\nb u\e+(u\e-1)\nb \chi|^p\dx\\
&\leq\int_{B_r}\eps^{p-d}(|\nb u\e|+|\nb \chi|)^p\dx\\
&\leq \int_{B_r}\eps^{p-d} |\nb u\e|^p\dx+C(r,\chi)\lt(E_{\eps,a}^{1-1/p}(\vartheta\e,v\e)\eps^\frac{p-d}{p}+\eps^{p-d}\rt)
\end{align*}
where we have used the inequality $(|a|+|b|)^p\leq |a|^p+C_p (|a|^{p-1}|b|+|b|^p)$ and Holder inequality. We get
\begin{equation}\label{eq: equivalence2}
\ov f^d_{\eps,a}(m,r)\leq E_{\eps,a}(\vartheta\e,v\e)\leq E_{\eps,a}(\vartheta\e,u\e)+O(\eps^{\frac{p-d}{p}})=f\e^{\tilde r}(m,r)+O(1)
\end{equation}

\medskip
\textit{Step 3: }
Combining inequalities~\eqref{eq: equivalence1} and \eqref{eq: equivalence2} we obtain $  f^d_{\eps,a}(m,r,\tilde r)- \ov f^d_{\eps,a}(m,r)= o(1).$
\end{proof}

\subsection{Study of the transition energy}\label{subSection: transitionenergy}
Given two values $r_1<r_2$ let us introduce the functional 
\begin{equation*}
\G(v;(r_1,r_2)):=\int_{r_1}^{r_2}t^{d-1}\lt[|v'|^p+(1-v)^2\rt]
\end{equation*}
and for any triplet $(\xi,\;r_1,\;r_2)\in [0,1]\times \R^+\times\R^+$ we set 
\begin{equation}\label{def: Gn}
 q^d(\xi,r_1,r_2):=\inf\lt\{\G(v;(r_1,r_2))\dt\;:\;v\in W^{1,p}(r_1,r_2),\;v(r_1)=\xi\mbox{ and } v(r_2)=1\rt\}.
\end{equation}
This value represents the cost of the transition from $\xi$ to $1$ in the ring $B_{r_2}\sm \ov B_{r_1}$. We will say that a function $v$ is admissible for the triplet $(\xi,r_1,r_2)$ if it is a competitor in the above minimization problem. Let us investigate the properties of the function introduced.
\begin{lemma}\label{lemmaA: propG}
For any fixed triplet $(\xi,r_1,r_2)\in [0,1]\times \R_+\times\R_+$ the infimum in equation~\eqref{def: Gn} is a minimum. Moreover there is a unique function achieving the minimum which is nondecreasing with range in the interval $[\xi,1]$. Finally the function $g$ satisfies the following properties
 \begin{enumerate}
   \item $r_2\mapsto q^d(\xi,r_1,r_2)$ is nonincreasing,
   \item $r_1\mapsto q^d(\xi,r_1,r_2)$ is nondecreasing,
  \item $\xi\mapsto q^d(\xi,r_1,r_2)$ is nonincreasing, and $g(1,r_1,r_2)=0$.
 \end{enumerate}
 Recalling the definition~\eqref{eq: ginfinito} of $q^d_{\oo}$, we have $q^d_{\oo}(\xi,\hat r)=q^d(\xi,r_1,\oo)$, and $q^d_{\oo}(0,0)>0$. Furthermore for any $r>0$ the map $\xi\mapsto q^d_{\oo}(\xi,r)$ is convex and continuous on $(0,+\oo)$.
\end{lemma}
\begin{proof} Let $(\xi,r_1,r_2)\in [0,1]\times \R_+\times\R_+$, the infimum is actually a minimum by means of the direct method of the calculus of variations. Such minimum is absolutely continuous on the interval $(r_1,r_2)$ by Morrey's inequality and is unique since $\G(v;(r_1,r_2))$ is strictly convex in $v$. Let $v\in W^{1,p}(r_1,r_2)$ be a minimizer of~\eqref{def: Gn} set 
  \begin{equation*}
  \ov v=\min\{\max( v,\xi),1\}
 \end{equation*} 
 then $\G(\ov v;(r_1,r_2))\leq\G(v;(r_1,r_2))$ if $v \neq \ov v$. As a consequence for every minimizer of~\eqref{def: Gn} we have $\xi\leq v\leq1$. Similarly setting 
 \[\ov v(s)=\max\{v(t):r_1\leq t\leq s\}\]
 we have  $\G(\ov v;(r_1,r_2))\leq\G(v;(r_1,r_2))$ if  $v \neq \ov v$. Hence $v$ is nondecreasing. 
   \begin{figure}[!ht]
  \centering
   \includegraphics[scale=0.4]{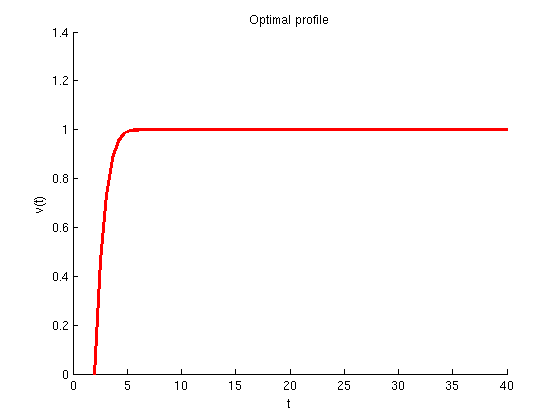}
   \caption{Graph obtained by a numerical optimization of problem~\eqref{def: Gn}, for the choice of the parameters $p=3$, $d=2$, $r_1=2$, $r_2=40$ and $\xi=0$.}\label{figura3}
  \end{figure}
 Let us now study the monotonicity of $g$. To do so let $v$ be the minimizer for $(\xi,r_1,r_2)$:
\begin{enumerate}
\item Let $\ov r_2> r_2$ and let us extend $v$ by $1$ on the interval $(r_2,\ov r_2)$. We have
\begin{equation*}
q^d(\xi,r_1,r_2)=\G(v;(r_1,r_2))=\G(v;(r_1,\ov r_2))\geq q^d(\xi,r_1,\ov r_2).
\end{equation*}
Hence $r_2\mapsto g$ is nonincreasing.
\item Let $0<\ov r_1< r_1$ and set $\Delta = r_1^{d}-\ov r_1^{d}>0$ and $\ov r_2=(r_2^{d}-\Delta)^{\frac{1}{d}}<r_2$. Define the diffeomorphism 
\begin{align}
\phi:(r_1,r_2)&\;\longrightarrow\;\;\;\; (\ov r_1,\ov r_2),\label{diffeomorphism}\\
s\;\;\;\;\;&\;\longmapsto\;\lt[s^{d}-  \Delta\rt]^{1/d}.\nonumber
 \end{align}
 Let $v$ be the minimizer of~\eqref{def: Gn} and $\ov v(s)=v\circ \phi(s)$. Let us remark that $\phi'(s)=s^{d-1}/\phi(s)^{d-1}$, thus it holds
  \begin{align*}
 q^d(\xi,r_1,r_2)&=\int_{r_1}^{r_2} t^{d-1}\lt[|v'|^p+(1-v)^2\rt]\dt= \int_{\ov r_1}^{\ov r_2} \phi(s)^{d-1}\lt[\frac{|\ov v'|^p}{|\phi'(s)|^p}+(1-\ov v)^2\rt]\phi(s)'\di s\\
 &= \int_{\ov r_1}^{\ov r_2} s^{d-1}\lt[\lt(1+\frac{\Delta}{s^{d}-\Delta}\rt)^{\frac{pd}{d}}|\ov v'|^p+(1-\ov v)^2\rt]\di s\geq q^d(\xi,\ov r_1,\ov r_2)\geq q^d(\xi,\ov r_1, r_2).
 \end{align*}
Therefore $r_1\mapsto q^d$ is nondecreasing.

\item Let $0\leq \xi< \ov \xi\leq 1$ and $v$ the absolutely continuous, nondecreasing minimizer of problem $q^d(\xi,r_1,r_2)$. Then there exists $\ov r\in (r_1,r_2)$ for which $v(\ov r)=\ov \xi$. Hence
\[q^d(\xi,r_1,r_2)\geq \G(v;(\ov r,r_2 ))\geq g(\ov \xi,\ov r,r_2)\geq g(\ov \xi, r_1,r_2).\]
Hence, $\xi\mapsto q^d$ is nonincreasing. Finally, for $\xi=1$ consider the constant function $v\equiv 1$ to get $g(1,r_1,r_2)=0$.
\end{enumerate}
Indeed, in view of the monotonicity, for every $r_1$ and $r_2$ we have
\[
g(0,r_1,r_2)\geq g(0,0,+\oo)=q^d_{\oo}(0,0).
\]
Let us show $q^d_{\oo}(0,0)>0$. As a matter of facts, taken the minimizer $v$ for the problem~\eqref{eq: ginfinito}, there exists $r\in(0,+\oo)$ such that $v(r)=1/2$ and we have
\[
q^d_{\oo}(0,0)\geq \int_{0}^{r} t^{d-1}\lt[|v'|^p+(1-v)^2\rt]\dt= \int_{0}^{r} t^{d-1}|v'|^p\dt+\frac{r^{d}}{4\,d}.
\]
A direct evaluation gives 
\[
\min\lt\{\int_{0}^{r} t^{d-1}|v'|^p\dt\,:\,v(r)=0\mbox{ and }v(r)=1/2\rt\}=\frac{c}{r}
\]
and we obtain the estimate
\[
q^d_{\oo}(0,0)\geq\frac{c}{r}+\frac{r^{d}}{4\,d}>0.
\]
Lastly, let us show that for any $r$ the function $q^d_{\oo}(\cdot,r)$  is convex. Consider two values $\xi_1,\,\xi_2\in (0,1)$ and the associated minimizers $v_1,\,v_2$ for the respective energy $q^d_{\oo}(\cdot,r)$. Indeed, for any $\lambda\in (0,1)$ the function $\lambda v_1+(1-\lambda )v_2$ is a competitor for the minimization problem $q^d_{\oo}(\lambda \xi_1+(1-\lambda)\xi_2,r)$, therefore it holds
\begin{align*}
q^d_{\oo}(\lambda \xi_1+(1-\lambda)\xi_2,r)&\leq \int_{r}^{\oo}t^{d-1}\lt[|\lambda v_1-(1-\lambda)v_2|^p+(1-\lambda v_1+(1-\lambda)v_2)^2\rt]\dt\\&\leq \lambda q^d_{\oo}( \xi_1,r)+(1-\lambda)q^d_{\oo}(\xi_2,r).
\end{align*}
Thus $q^d_{\oo}(\cdot,r)$ is continuous in the open interval $(0,1)$. To show the continuity in $0$ let $\xi$ be small and $v=\argmin q^d_{\oo}(\xi,r).$ Set 
\[
h(t):=\begin{dcases}
\frac{1}{1-\sqrt{\xi}}(t-\xi),&t<\sqrt{\xi},\\
t,& t\geq\sqrt{\xi}.
\end{dcases}
\]
and observe that $h\circ v$ is a competitor for the problem $q^d_{\oo}(0,r)$. Then
\begin{align*}
q^d_{\oo}(0,r)&\leq  \int_{r}^{\oo}t^{d-1}\lt[|(h\circ v)'|^p+(1- h\circ v^2\rt]\dt\\
&\leq \frac{1}{(1-\sqrt{\xi})^p}\;q^d_{\oo}(\xi,r)+\int_{r}^\oo t^{d-1}\lt[(1-h\circ v)^2-(1-v)^2\rt]\dt
\end{align*}
Let us estimate the second addend in the latter. By the definition of $f$ we have
\begin{align*}
\int_{r}^\oo t^{d-1}\lt[(1-h\circ v)^2-(1-v)^2\rt]\dt&=\int_{\{v<\sqrt{\xi}\}}t^{d-1}\lt[(1-h\circ v-v)^2(v-h\circ v)^2\rt]\dt\\
&\leq4\xi\;\int_{\{v<\sqrt{\xi}\}}t^{d-1}\dt\\
&\leq \frac{4\xi}{(1-\sqrt{\xi})^2}\;q^d_{\oo}(\xi,r).
\end{align*}
Since $q^d_{\oo}(\cdot, r)$ is monotone we have
\[|q^d_{\oo}(0,r)-q^d_{\oo}(\xi,r)|\leq\max\lt\{\frac{1-(1-\sqrt{\xi})^p}{(1-\sqrt{\xi})^p}\;,\;\frac{4\xi}{(1-\sqrt{\xi})^2}\rt\}\kappa,\]
which shows that $q^d_{\oo}(\cdot,r)$ is continuous in $0$.
\end{proof}

\subsection{Proof of Proposition~\ref{prop: structurefm}}\label{subSection: prop1}
We show that 
\begin{equation*}
\liminf_{\eps\dw0}\ov f^d_{\eps,a}(m,r)\geq f^d_a(m)
\end{equation*}
then equation~\eqref{eq: n-1liminf} easily follows from Lemma~\ref{lemmaA: equivalenceE}. For $m=0$ set $\vartheta=0$ and $u=1$, then $(\vartheta,u)\in Y_{\eps,a}(0,r)$ for any radius $r$ and $E_{\eps,a}(\vartheta,u;B_r)=0$ for each $\eps$. Now suppose $m>0$ and let $\xi\in (\eta,1)$. Consider the radially symmetric and monotone minimizing couple $(\vartheta\e,u\e)$ of Lemma~\ref{lemmaA: radial symmetry} and $r\x$ introduced in equation~\eqref{eq: rx}. Let us split the set of integration in the two sets $B_{r\x}$ and $B_r\sm B_{r\x}$,  we obtain
\begin{multline}\label{eq: 1exact1}
 \ov f^d_{\eps,a}(m,r)=E_{\eps,a}(\vartheta\e, u\e)\geq\\\underbrace{\int_{B_r\sm B_{r\x}}\lt[\eps^{p-d}|\nb u\e|^p+\frac{(1-u\e)^2}{\eps^{d}}\rt]
 \dx}_{a\e}+\underbrace{\int_{B_{r\x}}\frac{(1-u\e)^2}{\eps^{d}}\dx+\int_{B_{r}}\frac{u\e |\vartheta\e|^2}{\eps}\dx}_{b\e}.
\end{multline}
We deal with each addend separately. First observe that by Cauchy-Schwarz inequality, it holds 
\begin{equation*}
\frac{m^2}{\int_{B_{r}}\frac{1}{u\e}\dx}\;\leq\; \int_{B_{r}}u\e\vartheta\e^2\dx.
\end{equation*}
Plugging the latter in the term $b\e$ of \eqref{eq: 1exact1} we have
\begin{equation*}
b\e\geq\int_{B_{r\x}}\frac{(1-u\e)^2}{\eps^{d}}\dx+\frac{m^2}{\eps\lt(\int_{B_r\sm B_{r\x}}\frac{1}{u\e}\dx+\int_{B_{r\x}}\frac{1}{u\e}\dx\rt)}
\end{equation*}
taking into account $\eta\leq u\e\leq\xi$ in $B_{r\x}$, $\xi\leq u\e\leq 1$ in $B_r \sm B_{r\x}$ and $\eta=a\,\eps^{d+1}$ we obtain
\begin{equation}\label{eq: beps}
b\e\geq\om_d(1-\xi)^2\lt(\dfrac{r\x}{\eps}\rt)^{d}+\frac{m^2}{\dfrac{\om_d}{a}\lt(\dfrac{r\x}{\eps}\rt)^{d}+\om_d\dfrac{\eps r^{d}}{\xi}}.
\end{equation}
Since $b\e\leq \ov f^d_{\eps,a}(m,r) \leq C(m)$ we deduce that  $r\x/\eps$  belongs to a fixed compact subset $K=K(m,\xi)$ of $(0,+\oo)$. Up to extracting a subsequence, which we do not relabel, we can suppose $r\x/\eps$ to converge to some $\hat r>0$. Let us now consider the term $a\e$. Let $v\e$ be the radial profile of $u\e$ 
\begin{equation*}
 a\e=\int_{B_r\sm B_{r\x}}\lt[\eps^{p-d}|\nb u\e|^p+\frac{(1-u\e)^2}{\eps^{d}}\rt]\dx= (d-1)\;\om_{d}\;\int^{r/\eps}_{r\x/\eps}t^{d-1}\lt[|v\e'|^p+(1-v\e)^2\rt]\dt.
\end{equation*} 
With the notation introduced in Subsection~\ref{subSection: transitionenergy} and Lemma~\ref{lemmaA: propG} therein we deduce
\begin{equation*}	
\liminf_{\eps\dw0} a\e\geq (d-1)\,\om_{d}\,\liminf_{\eps\dw0} q^d\lt(\xi;\lt(r\x/\eps,r/\eps\rt)\rt) \geq  (d-1)\,\om_{d}\,q^d_{\oo}(\xi,\hat r),
\end{equation*}
where $q^d_{\oo}$ has been defined in~\eqref{eq: ginfinito}. Combining inequality~\eqref{eq: beps} and the latter we get
\begin{equation*}
\lim_{\eps\dw0}\ov f^d_{\eps,a}(m,r)\geq (d-1)\,\om_{d}\,q^d_{\oo}(\xi,\hat r)+(1-\xi)^2\;\om_d\;\hat r^{d}+\frac{a\;m^2}{\om_{d}\;\hat r^{d}}.
\end{equation*}
Sending $\xi$ to $0$ we have, by continuity (Lemma~\ref{lemmaA: propG}) $q^d_{\oo}(\xi,\hat r)\rw q^d_{\oo}(0,\hat r)$. Then taking the infimum in $\hat r$, we obtain
\[
\liminf_{\eps\dw0}\ov f^d_{\eps,a}(m,r)\geq\min_{\hat r}\lt\{ (d-1)\,\om_{d}\,q^d_{\oo}(0,\hat r)+\;\om_d\;\hat r^{d}+\frac{a\;m^2}{\om_{d}\;\hat r^{d}}\rt\}.
\]
Again by Lemma~\ref{lemmaA: propG} the function $q^d_{\oo}(0,\hat r)$ is nondecreasing in $\hat r$, and $q^d_{\oo}(0,0)>0$ therefore setting
\begin{equation*}
\kappa:=(d-1)\,\om_{d}\,q^d_{\oo}(0,0)\leq f^d_a(m)
\end{equation*}
we conclude the proof of Proposition~\ref{prop: structurefm}.

\subsection{Proof of Proposition~\ref{prop: flimsup}}\label{subSection: Prop2}
Let $\d>0$, by Lemma~\ref{lemmaA: propG} for $\eps$ sufficiently small 
\[
q^d(\eta;\lt(r_*,r/\eps)\rt) \leq q^d_{\oo}(0,r_*)+\d. 
\]
Let
\begin{equation*}
v_\d(t)=\argmin\lt\{\G\lt(v;\lt(r_*,\frac{r}{\eps}\rt)\rt)\dt\;:\;\;v\lt(r_*\rt)=\eta\mbox{ and } v\lt(\frac{r}{\eps}\rt)=1\rt\}.
\end{equation*}
and set
\begin{equation*}
u\e(t):=
\begin{dcases}
\eta&\mbox{ for } 0\leq t\leq  r_*\eps\\
v_\d\lt(\frac{t}{\eps}\rt)&\mbox{ for } r_*\eps\leq  t\leq r\\
\end{dcases}
\end{equation*}
Set $\vartheta\e(s)$ to be constant equal to $\frac{m}{\;\om_{d}\;   (\eps r_*)^{d}}$ on the ball $B_{\eps r_*}$ and zero outside. Indeed, the couple $(\vartheta\e,u\e(|x|))$  belongs to $\ov Y_{\eps,a}(m,r)$. That is because $u\e$ is greater then $\eta$ and attains value $1$ at the border of $B_r$ and
\begin{equation*}
\int_{B_r} \vartheta\e(x)\dx=\frac{m}{\om_d (\eps r_*)^{d}}\;\om_d(\eps r_*)^d=m.
\end{equation*}
Let us show that the couple $(\vartheta\e,u\e)$ defined satisfy inequality~\eqref{eq: n-1recovery}. Taking advantage of the radial symmetry of the functions we get
\begin{multline*}
 E_{\eps,a}(\vartheta\e,u\e)=\int_{\eps  r_*}^r t^{d-1} \lt[\eps^{p+d}|u\e'|^p+\frac{(1-u\e)}{\eps^{d}}\rt]\dt\\+\frac{(1-\eta)^{2}}{\eps^{d}}\;\om_{d}\;(\eps r_*)^{d}+\frac{\eta}{\eps}\lt(\frac{m}{\;\om_{d}\;(\eps r_*)^{d}}\rt)^{2} \;\om_{d}\;(\eps  r_*)^{d}.
\end{multline*}
 By simplifying the expression and considering the change of variable $s=\frac{t}{\eps}$ in the latter it holds
\begin{align*}
   E_{\eps,a}(\vartheta\e,u\e)&=(d-1)\;\om_{d}\;\int_{r_*}^{\frac{ r}{\eps}} s^{d-1} \lt[|v_\d'|^p+(1-v_\d)\rt]\di s+(1-\eta)^{2}\;\om_{d}\; r_\d^{d}+\frac{\eta}{\eps^{d+1}}\frac{\;m^2}{\;\om_{d}\; r_*^{d}}\\
  &\leq (d-1)\;\om_{d}\;q^d\lt(\eta;(r_*,r/\eps)\rt)+(1-\eta)^{2}\;\om_{d}\; r_*^{d}+\frac{\eta}{\eps}\frac{\;m^2}{\;\om_{d}\; r_*^{d}}
\end{align*}
Then, by Lemma~\ref{lemmaA: propG} for $\eps$ sufficiently small we have
\begin{equation*}
E_{\eps,a}(\vartheta\e,u\e)\leq \frac{a \,m^2}{\;\om_{d}\;r_*^{d}}+\;\om_{d}\; r_*^{d}+(d-1)\;\om_{d}\;q^d_{\oo}(0, r_*)+(d-1)\om_{d-1}\d=f^d_a(m)+C\d,
\end{equation*}
which ends the proof of Proposition~\ref{prop: flimsup}.

\subsection{Proof of Proposition~\ref{prop: fproprieties}}\label{subSection: fproprieties}
Propositions~\ref{prop: structurefm},~\ref{prop: flimsup} and lemma~\ref{lemmaA: equivalenceE} ensure that 
\begin{equation}
f^d_a(m)=\lim_{\eps\dw0}\ov f^d_{\eps,a}(m,r)= \lim_{\eps\dw0} f^d_{\eps,a}(m,r,\tilde r)
\end{equation}
independently of the choices for $r$ and $\tilde r<r$. For the sake of clarity we introduce 
\begin{equation*}
  T(m,r):=\lt\{\frac{a \,m^2}{\;\om_{d}\; r^{d}}+\;\om_{d}\; r^{d}+(d-1)\;\om_{d}\;q^d_{\oo}(0, r)\rt\}
\end{equation*}
and recall that $f^d_a(m)= \min_r T(m,r)$ for $m>0$ and $ f^d_a(0)=0$, see~\eqref{Def: reducedfd}.
\begin{proof}~\\
Let us prove the continuity of $f^d_a$ on $(0,+\oo)$. For $m_1, m_2\in (0,+\oo)$ and for $i=1,2$ let $r_i$ be such that $f^d_a(m_i)=T(m_i,r_i)$. On one hand comparing with $r=1$ it holds
\begin{equation}\label{ineqcontinuity}
\frac{m_i^2}{\om_{d-1}\,r_i^d}\leq f^d_a(m_i)\leq T(m_i,1)
\end{equation}
on the other hand analougusly we have
\begin{equation}
\om_{d-1}\,r_i^d\leq f^d_a(m_i)\leq T(m_i,1).
\end{equation}
Consequently  $\om_{d-1} \,r_i^d$ belongs to the compact set $[m_i/T(m_i,1),T(m_i,1)]$. Now remark that 
\begin{equation*}
f^d_a(m_1)\leq T(m_1,r_2)=f^d_a(m_2)+T(m_1,r_2)-T(m_2,r_2)
\end{equation*}
thus 
 \begin{equation*}
 |f^d_a(m_1)-f^d_a(m_2)|\leq|T(m_1,r_2)-T(m_2,r_2)|\leq \frac{|m_1^2-m_2^2|}{\om_{d-1}\min\{r_1^d.r_2^d\}}
 \end{equation*}
 and taking into account inequality~\eqref{ineqcontinuity} we have
  \begin{equation*}
 |f^d_a(m_1)-f^d_a(m_2)|\leq (m_1+m_2)\max\lt\{\frac{T(m_1,1)}{m_1^2} ,\frac{T(m_2,1)}{m_2^2} \rt\}|m_1-m_2|.
 \end{equation*}
Observing  that $T(\cdot,1)$ is continuous  we conclude that $f^d_a$ is continuous on $(0,+\oo)$.

\noindent
Next, we see that $f^d_a$ is non decreasing. Let $0<m_1< m_2$ and $r>0$. Let $(\vartheta,u)\in \ov Y_{\eps,a}(m_2,r)$ such that $E_{\eps,a}\lt(\vartheta,u;B_r\rt)=\ov f^d_{\eps,a}(m_2,r)$ . Set $\ov \vartheta= m_1\vartheta/m_2$ and remark that the couple $(\ov\vartheta,u)$ belongs to $\ov Y_{\eps,a}(m_1,r)$. Therefore we have the following set of inequalities
 \begin{equation*}
  \ov f^d_{\eps,a}(m_1,r)\leq E_{\eps,a}(\ov\vartheta,u;B_r)=E_{\eps,a}\lt(\frac{m_1\vartheta}{m_2},u;B_r\rt)< E_{\eps,a}\lt(\vartheta,u;B_r\rt)=\ov f^d_{\eps,a}(m_2,r).
 \end{equation*}
Passing to the limit as $\eps\dw0$ we obtain
\begin{equation*}
f^d_a(m_1)\leq f^d_a(m_2).
\end{equation*}
Let us now prove the sub-additivity. For a radius $r$ consider the competitors $(\vartheta_j,u_j)\in \ov Y_{\eps,a}(m_j,r)$ for $j=1,2$. Consider the ball $B_{2r+1}$ centered in the origin and two points $x_1,x_2$ such that the balls $B_{r}(x_1)$, $ B_r(x_2)$ are disjoint and contained in $B_{2r+1}$. Set 
 \begin{equation*}
 \ov\vartheta(x):=
  \begin{dcases}
   \vartheta_1(x-x_1),&x\in B_{r}(x_1),\\
      \vartheta_2(x-x_2),&x\in B_{r}(x_2),\\
      \qquad0, &\mbox{otherwise,}
  \end{dcases}
\qquad\mbox{and}\qquad
    \ov u(x):=
  \begin{dcases}
   u_1(x-x_1),&x\in B_{r}(x_1),\\
      u_2(x-x_2),&x\in B_{r}(x_2),\\
      \qquad1, &\mbox{otherwise,}
  \end{dcases}
 \end{equation*}
and observe that the couple $(\ov\vartheta,\ov u)$ belongs to $\ov Y(m_1+m_2,2r+1)$. Being the balls $B_{r}(x_j)$ disjoint we have
  \begin{align*}
  \ov f^d_{\eps,a}(m_1+m_2,r_1+r_2)&\leq E_{\eps,a}(\vartheta_1(x-x_1),u_1(x-x_1);B_r(x_1))+ E_{\eps,a}(\vartheta_2(x-x_2),u_2(x-x_2);B_r(x_2))\\
  &=\ov f^d_{\eps,a}(m_1,r)+f^d_{\eps,a}(m_2,r).
 \end{align*}
Passing to the limit as $\eps\dw0$, and recalling that it is independent of the choice of the radius, we get
 \begin{equation*}
  f^d_{a}(m_1+m_2)\leq f^d_{a}(m_1)+f^d_{a}(m_2).
 \end{equation*}
\end{proof}

We conclude the appendix by showing that
\begin{lemma}\label{lemmaA:pointwiseconvergence}
For any sequence $a_i\dw0$ it holds 
\begin{equation*}
f^d_{a_i} \longrightarrow \kappa \mathbf{1}_{(0,\oo)}
\end{equation*}
pointwise.
\end{lemma}
\begin{proof}
 We have already shown that $f^d_a(m)\geq \kappa$ for $m>0$. For $m>0$ choose $\hat r=(\sqrt{a}m)^{1/d}$, then by definition it holds 
 \[
 \kappa\leq f^d_a(m)\leq (d-1)\,\om_{d}\,q^d_{\oo}(0,(\sqrt{a}m)^{1/d})+\om_d\sqrt{a}m+\frac{\sqrt{a}m}{\om_d}.
 \]
 Finally simply recall that $(d-1)\,\om_{d}\,q^d_{\oo}(0,0)=\kappa$ and that $q^d_{\oo}(0,\cdot)$ is continuous.
\end{proof}
\bibliography{bib}
\end{document}